\documentclass{article}
\usepackage[utf8]{inputenc}
\usepackage[left=2.5cm, right=2.5cm, top=2cm]{geometry}
\usepackage{graphicx}
\usepackage{subcaption}
\usepackage{amsmath}
\usepackage{amsthm}
\usepackage{amsfonts} 
\usepackage{algorithm}
\usepackage[noend]{algpseudocode}
\usepackage{bbm}

\usepackage{url}
\usepackage{amsmath,amssymb,amsbsy}
\usepackage{mathdots}
\usepackage{paralist}
\usepackage{xcolor}
\usepackage{color}
\usepackage{graphicx}
\usepackage{algorithm,algpseudocode}
\usepackage{comment}
\usepackage{fancyhdr}
\usepackage{cite}
\usepackage{cleveref}

\newtheorem{lemma}{Lemma}
\newtheorem{theorem}{Theorem}

\newcommand{\R}{\mathbb{R}}

\renewcommand{\P}{\mathbb{P}}
\newcommand{\E}{\mathbb{E}}

\newcommand{\vct}[1]{\boldsymbol{#1}}
\newcommand{\mtx}[1]{\boldsymbol{#1}}
\newcommand{\bvct}[1]{\mathbf{#1}}

\newcommand{\rank}{\operatorname{rank}}

\newcommand{\calT}{\mathcal{T}}

\newcommand{\calX}{\mathcal{X}}

\newcommand{\va}{\vct{a}}
\newcommand{\vb}{\vct{b}}

\newcommand{\ve}{\vct{e}}

\newcommand{\vq}{\vct{q}}

\newcommand{\vu}{\vct{u}}
\newcommand{\vv}{\vct{v}}

\newcommand{\vz}{\vct{z}}

\newcommand{\mA}{\mtx{A}}
\newcommand{\mB}{\mtx{B}}
\newcommand{\mC}{\mtx{C}}

\newcommand{\mM}{\mtx{M}}

\newcommand{\mQ}{\mtx{Q}}
\newcommand{\mR}{\mtx{R}}

\newcommand{\mU}{\mtx{U}}
\newcommand{\mV}{\mtx{V}}
\newcommand{\mW}{\mtx{W}}
\newcommand{\mX}{\mtx{X}}
\newcommand{\mY}{\mtx{Y}}
\newcommand{\mZ}{\mtx{Z}}

\newcommand{\mSigma}{\mtx{\Sigma}}

\newcommand{\mId}{{\bf I}}

\graphicspath{{./figs/}}

\newlength{\imgwidth}
\newboolean{twoColVersion}
\setboolean{twoColVersion}{false}
\newcommand{\twoCol}[2]{\ifthenelse{\boolean{twoColVersion}} {#1} {#2} }

\title{Tensor Deli: Tensor Completion for Low CP-Rank Tensors via Random Sampling}

\author{Cullen Haselby, Mark Iwen, Santhosh Karnik, and Rongrong Wang \thanks{Cullen Haselby, Mark Iwen, and Rongrong Wang are with the Department of Mathematics at Michigan State University. Mark Iwen, Santhosh Karnik, and Rongrong Wang are with the Department of Computational Mathematics, Science, and Engineering at Michigan State University. (e-mail: \hbox{haselbyc@msu.edu}, \hbox{iwenmark@msu.edu}, \hbox{karniksa@msu.edu}, \hbox{wangron6@msu.edu}). Mark Iwen and Cullen Haselby were both supported in part by NSF DMS 2106472. Rongrong Wang was supported in part by NSF CCF-2212065.  This work is an extension of \cite{haselby2023tensor}.}}

\begin{document}

\maketitle

\begin{abstract}
We propose two provably accurate methods for low CP-rank tensor completion - one using adaptive sampling and one using nonadaptive sampling. Both of our algorithms combine matrix completion techniques for a small number of slices along with Jennrich's algorithm to learn the factors corresponding to the first two modes, and then solve systems of linear equations to learn the factors corresponding to the remaining modes. For order-$3$ tensors, our algorithms follow a ``sandwich'' sampling strategy that more densely samples a few outer slices (the bread), and then more sparsely samples additional inner slices (the bbq-braised tofu) for the final completion. For an order-$d$, CP-rank $r$ tensor of size $n \times \cdots \times n$ that satisfies mild assumptions, our adaptive sampling algorithm recovers the CP-decomposition with high probability while using at most $O(nr\log r + dnr)$ samples and $O(n^2r^2+dnr^2)$ operations.  Our nonadaptive sampling algorithm recovers the CP-decomposition with high probability while using at most $O(dnr^2\log n + nr\log^2 n)$ samples and runs in polynomial time.  Numerical experiments demonstrate that both of our methods work well on noisy synthetic data as well as on real world data.
\end{abstract}

\section{Introduction}

Consider the problem of recovering an order-$d$ tensor $\calT \in \R^{n_1 \times \cdots \times n_d}$ after observing a subset of its entries. Of course, without any structural assumptions on the tensor, this is impossible. One possible structural assumption is that the tensor has a low CP-rank, i.e., that the tensor is a sum of a small number of rank one tensors. Note that for $d = 2$ this reduces to the well-studied low-rank matrix completion problem. Similar to matrix completion, tensor completion has been used in a variety of applications including recommender systems\cite{frolov2017tensor,zhu2018fairness,nguyen2023tensor}, hyperspectral imaging\cite{wang2020anomaly, lin2023hyperspectral, zhang2021learning}, visual data\cite{liu2012tensor,liu2019low}, and more \cite{tan2016short, xie2016accurate, trickett2013interpolation}. In these applications, the data can be modeled as a low rank tensor, but the data is either incomplete or expensive to acquire. As such, it is useful to be able to learn the low-rank structure of the tensor from as few observed entries as possible.

\subsection{Notation and Preliminaries}
A CP-decomposition of an order-$d$ tensor $\calT \in \R^{n_1 \times \cdots \times n_d}$ is a decomposition of the form \[\calT = \sum_{\ell = 1}^{r}\va^{(1)}_{\ell} \circ \cdots \circ \va^{(d)}_{\ell},\] where $\va^{(k)}_{\ell} \in \R^{n_k}$ for $k \in [d]$ and $\ell \in [r]$, and $\circ$ denotes the outer product. The CP-rank of a tensor is the minimum $r$ such that such a decomposition exists. It is often convenient to define the factor matrices $\mA^{(1)},\ldots,\mA^{(d)}$ of the CP-decomposition by \[\mA^{(k)} = \begin{bmatrix}\va^{(k)}_1 & \cdots & \va^{(k)}_r\end{bmatrix} \in \R^{n_k \times r} \quad \text{for} \quad k \in [d],\] and use the notation \[\calT = \left[\left[\mA^{(1)},\ldots,\mA^{(d)}\right]\right]\] to indicate that $\calT$ has a CP-decomposition using the columns of the factor matrices $\mA^{(1)},\ldots,\mA^{(d)}$. In \cite{bhaskara2014uniqueness} it is shown that if the sum of the Kruskal ranks of $\mA^{(1)},\ldots,\mA^{(d)}$ are at least $2r+d-1$, then the CP-decomposition of $\calT$ is unique up to permuting the order of the rank one terms and rescaling each vector such that the product of the scale factors for each rank one term is $1$. 

\subsection{Prior Results}
\subsubsection{Low-Rank Matrix Completion}
Many works on low-rank matrix completion make a coherence assumption on the rowspace and/or column space of the matrix to ensure that the $\ell_2$-norm energy in the matrix isn't concentrated in a few entries. Specifically, for any $r$-dimensional subspace $U \subset \R^n$, its coherence $\mu(U)$ is defined by \[\mu(U) := \dfrac{n}{r}\max_{i \in [n]}\|\text{proj}_{U}\ve_i\|_2^2,\] where $\ve_i \in \R^n$ for $i \in [n]$ denotes the $i^{\rm th}$ standard basis vector. We will also define the coherence of a $n \times r$ matrix to be the coherence of its columnspace. 

Chen\cite{chen2015incoherence} showed that if a rank-$r$ matrix $\mM \in \R^{n_1 \times n_2}$ has rowspace and columnspace coherences bounded by $\mu_0$, and each entry is sampled independently and uniformly at random with probability $c_0\tfrac{\mu_0 r\log^2(n_1+n_2)}{\min\{n_1,n_2\}}$, then with probability at least $1-c_1(n_1+n_2)^{-c_2}$, a nuclear norm minimization problem will recover $\mM$. 
Balcan and Zhang \cite{BalcanZhang2016} showed that if a rank-$r$ matrix $\mM \in \R^{n_1 \times n_2}$ has columnspace coherence bounded by $\mu_0$, then an adaptive sampling scheme allows one to complete the matrix after observing at most $O(\mu_0n_2r\log(r/\delta))+n_1r$ entries.

\subsubsection{Low CP-Rank Tensor Completion}

Many prior works on low CP-rank tensor completion use non-adaptive and uniform sampling \cite{cai2019nonconvex,jain2014provable,barak2016noisy,yuan2016tensor,potechin2017exact,montanari2018spectral,liu2020tensor,kivva2020exact,stephan2023non}, i.e., each entry of the tensor is sampled independently with some constant probability. For order $d = 3$ tensors, some of these works \cite{cai2019nonconvex,barak2016noisy,kivva2020exact} can handle CP-ranks up to roughly $n^{3/2}$, but all of these works require at least $\Omega(n^{3/2})$ samples, even when the CP-rank is $r = O(1)$. Of these works, only \cite{montanari2018spectral,stephan2023non} address higher order tensors (i.e. $d \ge 4$), but both still require at least $\Omega(n^{d/2})$ samples even when $r = O(1)$.

Krishnamurthy and Singh \cite{krishnamurthy2013low} showed that for an order-$d$ tensor $\calT \in \R^{n \times \cdots \times n}$, if its mode-$1,...,d-1$ fiberspaces have coherences bounded by $\mu_0$, then an adaptive sampling and reconstruction algorithm can recover the CP-decomposition after observing $O(d^2\mu_0^{d-1} nr^{d-1/2}\log(dr/\delta))$ samples with probability at least $1-\delta$.  In \cite{bhojanapalli2015new} an adaptive sampling algorithm that provably
recovers a symmetric rank-$r$ third order tensor from $O(nr^3\log^2 n)$ samples is proposed.
In contrast, lower sample complexity has been discovered for tucker rank based tensor completion. In particular, in  \cite{zhang2019cross} it is shown that an $n\times n\times n$ third-order tensor of Tucker rank-$(r,r,r)$
 can be recovered from as few as
$O(r^3 + rn)$ noiseless measurements. 

\subsubsection{Jennrich's Algorithm}

Jennrich's algorithm \cite{Leurgans1993} is an algorithm for recovering the CP-decomposition of an order-$3$, CP-rank $r$ tensor $\calX \in \R^{n_1 \times n_2 \times n_3}$ given all entries. Jennrich's algorithm works due to the fact that every mode-$3$ slice is simultaneously diagonalized by the mode-$1$ and mode-$2$ factor matrices. Specifically, if $\calX = [[\mA,\mB,\mC]]$, then Jennrich's algorithm first draws two random vectors $\vu,\vv \in \R^{n_3}$ and forms two random linear combinations of the mode-$3$ slices $\mX_{\vu} = \sum_{i_3 = 1}^{n_3}\vu_{i_3}\calX_{:,:,i_3}$ and $\mX_{\vv} = \sum_{i_3 = 1}^{n_3}\vv_{i_3}\calX_{:,:,i_3}$. Then, Jennrich's algorithm computes the eigendecompositions of $\mX_{\vu}\mX_{\vv}^{\dagger}$ and $\mX_{\vv}\mX_{\vu}^{\dagger}$. It can be shown that if $\rank(\mA) = \rank(\mB) = r$ and $\mC$ has Kruskal-rank $\ge 2$, then with probability $1$, the eigenvectors of $\mX_{\vu}\mX_{\vv}^{\dagger}$ are $\{\va_{\ell}\}_{\ell = 1}^{r}$, the eigenvectors of $\mX_{\vv}\mX_{\vu}^{\dagger}$ are $\{\vb_{\ell}\}_{\ell = 1}^{r}$, and the eigenvalue of $\mX_{\vu}\mX_{\vv}^{\dagger}$ corresponding to $\va_{\ell}$ is the reciprocal of the eigenvalue of $\mX_{\vv}\mX_{\vu}^{\dagger}$ corresponding to $\vb_{\ell}$. Hence, Jennrich's algorithm can recover the columns of $\mA$ and $\mB$, but also pair them correctly. If needed, $\mC$ can then be recovered by solving a system of linear equations.

\subsection{Main Contributions}

In this paper, we introduce two algorithms with rigorous theoretical guarantees for recovering the CP-decomposition of an order-$d$, CP-rank $r$ tensor after observing a subset of its entries - one using adaptive sampling and another using nonadaptive sampling. Both of our algorithms involve densely sampling a few mode-$(1,2)$ slices of the tensor and sparsely sampling other subtensors in a way such that the total number of samples is only slightly worse than linear in each of the dimensions of the tensor. Also, both of our algorithms work by using a matrix completion algorithm on the densely sampled mode-$(1,2)$ slices, using Jennrich's algorithm with the completed slices to learn the first two modes of the CP-decomposition, and then solving systems of linear equations to learn the remaining modes.

With mild assumptions (discussed in Section~\ref{sec:MainResults}), our adaptive sampling algorithm recovers the CP-decomposition of an order-$d$, CP-rank $r$ tensor $\calT \in \R^{n \times \cdots \times n}$ using $O(\mu_0 nr\log(r/\delta)+dnr)$ samples with probability $\ge 1-\delta$. This is a noticeable improvement over the adaptive sampling algorithm in \cite{krishnamurthy2013low} which requires $O(d^2\mu_0^{d-1} nr^{d-1/2}\log(dr/\delta))$ samples, but uses slightly different assumptions. It is also only a mild log-factor worse than the information theoretic bound of $O(dnr)$ samples.  Furthermore, our algorithm runs in $O(n^2r^2 + dnr^2)$ operations.

Additionally, with mild assumptions (also discussed in Section~\ref{sec:MainResults}), our nonadaptive random sampling algorithm recovers the CP-decomposition of an order-$d$, CP-rank $r$ tensor $\calT \in \R^{n \times \cdots \times n}$ using $O(d\mu_0^2nr^2\log n + \mu_0 nr\log^2 n)$ samples with high probability. This is a significant improvement over nonadaptive and uniformly random sampling which requires at least $\Omega(n^{d/2})$ samples. Furthermore, our algorithm runs in polynomial time. Specifically, the dominating cost of this algorithm is solving a few nuclear norm minimization problems to complete the densely sampled mode-$(1,2)$ slices.

In Section~\ref{sec:Order3Tensors}, we build up intuition for our algorithms by presenting simpler versions to handle order-$3$ tensors. In Section~\ref{sec:MainResults}, we present our main results for order-$d$ tensors where the factor matrices corresponding to modes $3,\ldots,d$ have no entries which are exactly zero. In Section~\ref{sec:GeneralResults}, we present results for the most general case where the tensor has order-$d$ and each column of the factor matrices corresponding to modes $3,\ldots,d$ is allowed to have a specified fraction of its entries exactly equal to zero. In Section~\ref{sec:Experiments}, we present numerical experiments on both synthetic data and real world data which demonstrates that our algorithm is robust to noise. The proofs of our results are in the Appendix.

\section{Order-3 Tensors}
\label{sec:Order3Tensors}
In order to build intuition for our algorithms, we will first consider the simplest case where our tensor has order $d = 3$. In both our adaptive and non-adaptive algorithms, our algorithm will involve densely sampling $s \ll n_3$ mode-$3$ slices of the tensor and sparsely sampling the remaining slices. See Figure~\ref{fig:SamplingPattern} for a depiction of the sampling patterns. Also, the reconstruction method for both algorithms follows the outline below: 
\begin{enumerate}
    \item Use a matrix completion algorithm on the $s$ densely sampled slices $\calT_{:,:,i_3}$, $i_3 \in S$.
    \item Use Jennrich's algorithm on the completed $n_1 \times n_2 \times s$ subtensor to learn $\mA^{(1)}$ and $\mA^{(2)}$.
    \item Solve systems of linear equations to determine $\mA^{(3)}$.
\end{enumerate}
Our adaptive and nonadaptive sampling algorithms will differ in how many samples are taken, when the samples are taken, as well as which matrix completion algorithm is used in step 1.

\begin{figure}[H]
\centering
    \includegraphics[scale = 0.225]{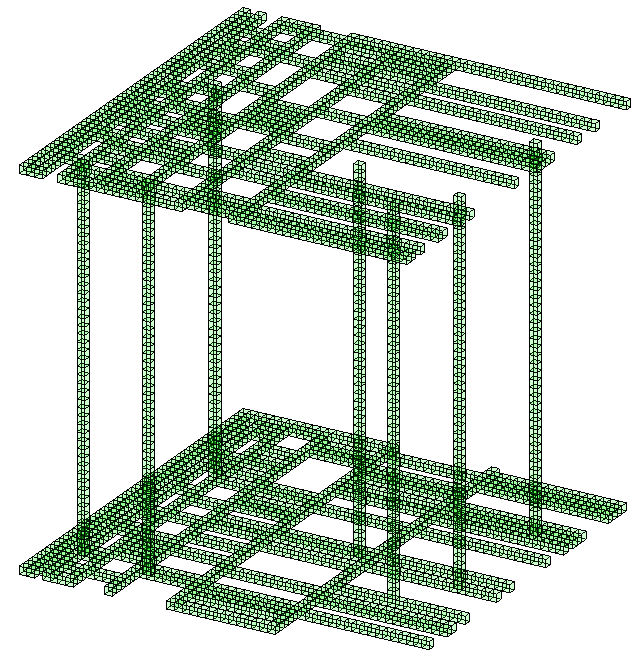} \hspace{1 in} \includegraphics[scale = 0.225]{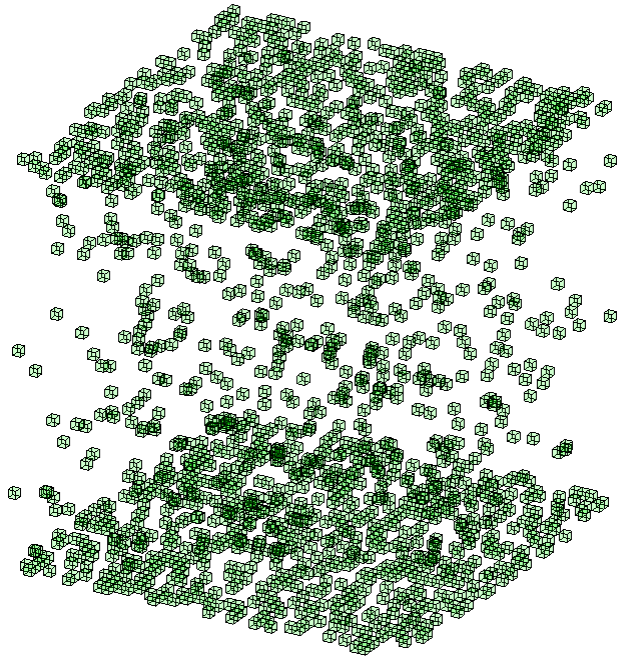}
    \caption{Illustration of the sampling patterns of our adaptive tensor sandwich (left) and nonadaptive tensor sandwich (right) algorithms}
    \label{fig:SamplingPattern}
\end{figure}

We now state the simplified version of our adaptive Tensor Sandwich algorithm for order-$3$ tensors as well as a theorem with the theoretical guarantees.

\begin{algorithm}[H]
\begin{algorithmic}[1]
\caption{Adaptive Tensor Sandwich for Order-$3$ tensors}
\label{alg:3mode_adapt}
    \State Pick a subset $S \subset [n_3]$ with $|S| = s$ indices.
    \For{$i_3 \in S$}
        \State Use the algorithm in \cite{BalcanZhang2016} to sample and complete $\calT_{:,:,i_3}$.
    \EndFor
    \State $[\mA^{(1)},\mA^{(2)}] = \text{Jennrich}(\calT_{:,:,S})$ 
    \State Perform QR decomposition with pivoting on $(\mA^{(1)} \odot \mA^{(2)})^T$ to identify a subset of $r$ linearly independent rows of $\mA^{(1)} \odot \mA^{(2)}$. Let $L \subset [n_1] \times [n_2]$ be the subset of indices $(i_1,i_2)$ corresponding to these rows.
    \For{$i_3 \in [n]$}
        \State Sample $\calT_{i_1,i_2,i_3}$ for $(i_1,i_2) \in L$
        \State Solve the system of equations for $\{\mA^{(3)}_{i_3,\ell}\}_{\ell = 1}^{r}$ \vspace{-0.1 in} $$\sum_{\ell = 1}^{r}\mA^{(1)}_{i_1,\ell}\mA^{(2)}_{i_2,\ell}\mA^{(3)}_{i_3,\ell} = \calT_{i_1,i_2,i_3} \quad \text{for} \quad (i_1,i_2) \in L.$$ \vspace{-0.1 in}
    \EndFor
    \State \Return $\mA^{(1)},\mA^{(2)},\mA^{(3)}$
\end{algorithmic}
\end{algorithm}

\begin{theorem} 
\label{thm:3mode_adapt}
Suppose that the factor matrices satisfy the following assumptions 
\begin{enumerate}
\item[(a)] $\mu(\mA^{(1)}) \le \mu_0$
\item[(b)] $\rank(\mA^{(1)}) = \rank(\mA^{(2)}) = r$    
\item[(c)] every $s \times r$ submatrix of $\mA^{(3)}$ has Kruskal rank $\ge 2$
\end{enumerate}
Then, with probability at least $1-s\delta$, Algorithm~\ref{alg:3mode_adapt} completes $\calT$ and uses at most \[O\left(s\mu_0 n_2r\log(r/\delta)\right) + sn_1r + n_3r\] samples.
\end{theorem}

We note that assumption (a) is a standard incoherence assumption that is made by most matrix and tensor completion results, while assumptions (b) and (c) are precisely the necessary assumptions for Jennrich's algorithm to work on the subtensor $\calT_{:,:,S}$. Also, if $\mA^{(3)}$ is drawn randomly from a continuous distribution, assumption (c) with $s = 2$ holds with probability $1$. Hence, for a typical $n \times n \times n$ tensor, Algorithm~\ref{alg:3mode_adapt} completes $\calT$ with probability at least $1-2\delta$ using $O(\mu_0 nr\log(r/\delta))$ samples. 

We also state the simplified version of our nonadaptive Tensor Sandwich algorithm for order-$3$ tensors as well as a theorem with the theoretical guarantees.

\begin{algorithm}[H]
\begin{algorithmic}[1]
\caption{Nonadaptive Tensor Sandwich for Order-$3$ tensors}
\label{alg:3mode_nonadapt}
    \State Pick a subset $S \subset [n_3]$ with $|S| = s$ indices.
    \State Form a subset $\Omega \subset [n_1] \times [n_2] \times [n_3]$ by independently including each entry $(i_1,i_2,i_3) \in [n_1] \times [n_2] \times [n_3]$ in $\Omega$ with probability \vspace{-0.1 in} $$\begin{cases}c_0\dfrac{\mu_0 r\log^2(n_1+n_2)}{\min\{n_1,n_2\}} & \text{if } i_3 \in S \\  \\ c_3\dfrac{\mu_0^2 r^2\log n_3}{n_1n_2} & \text{if } i_3 \not\in S \end{cases}$$ \vspace{-0.1 in}
    \For{$i_3 \in S$}
        \State Use nuclear norm minimization complete $\calT_{:,:,i_3}$.
    \EndFor
    \State $[\mA^{(1)},\mA^{(2)}] = \text{Jennrich}(\calT_{:,:,S})$ 
    \For{$i_3 \in [n]$}
        \State Solve the system of equations for $\{\mA^{(3)}_{i_3,\ell}\}_{\ell = 1}^{r}$ \vspace{-0.1 in} $$\sum_{\ell = 1}^{r}\mA^{(1)}_{i_1,\ell}\mA^{(2)}_{i_2,\ell}\mA^{(3)}_{i_3,\ell} = \calT_{i_1,i_2,i_3} \quad \text{for} \quad (i_1,i_2) \text{ s.t. } (i_1,i_2,i_3) \in \Omega.$$ \vspace{-0.1 in}
    \EndFor
    \State \Return $\mA^{(1)},\mA^{(2)},\mA^{(3)}$
\end{algorithmic}
\end{algorithm}

\begin{theorem}
\label{thm:3mode_nonadapt}
Suppose that the factor matrices satisfy the following assumptions
\begin{enumerate}
\item[(a)] $\mu(\mA^{(1)}) \le \mu_0$ and $\mu(\mA^{(2)}) \le \mu_0$
\item[(b)] $\rank(\mA^{(1)}) = \rank(\mA^{(2)}) = r$
\item[(c)] every $s \times r$ submatrix of $\mA^{(3)}$ has Kruskal rank $\ge 2$
\end{enumerate}
Then, Algorithm~\ref{alg:3mode_nonadapt} completes $\calT$ with probability at least $1-c_1s(n_1+n_2)^{-c_2}-rn_3^{-(c_3-1)}$. Furthermore, with high probability, the number of randomly drawn samples is at most \[O\left(s\mu_0r\max\{n_1,n_2\}\log^2(n_1+n_2) + \mu_0^2 r^2n_3\log n_3\right).\] Here, $c_0, c_1, c_2 > 0$ are the constants in \cite{chen2015incoherence}, and $c_3$ can be any number larger than $1$.
\end{theorem}

We note that the price of having nonadaptive sampling is that we need more samples in the densely sampled slices so that nuclear norm minimization completes the slices, as well as in the sparsely sampled slices so that the systems of equations for the mode-$3$ factor matrix are fully determined. Additionally, we need a coherence assumption on both $\mA^{(1)}$ and $\mA^{(2)}$ instead of just $\mA^{(2)}$ so that the assumptions of the result in \cite{chen2015incoherence} hold. Again, for a typical $n \times n \times n$ tensor, the third assumption is satisfied for $s = 2$, and thus, Algorithm~\ref{alg:3mode_nonadapt} completes $\calT$ with high probability using $O(\mu_0 nr\log^2 n + \mu_0^2 n r^2 \log n)$ samples. By sampling $s$ slices of the tensor densely and $n-s$ slices of the tensor sparsely, we are able to reduce the sample complexity from having an $n^{3/2}$ dependence on $n$ to having a slightly worse than linear dependence on $n$.

\section{Main Results}
\label{sec:MainResults}
We will now consider the case where the tensor has order-$d$, but the factor matrices $\mA^{(3)},\ldots,\mA^{(d)}$ have no entries which are exactly zero (We relax this assumption in Section~\ref{sec:GeneralResults}). This assumption has a couple important consequences. First, any mode-$(1,2)$ slice is rank-$r$. Hence, performing Jennrich's algorithm on any mode-$(1,2,3)$ subtensor of size $n_1 \times n_2 \times s$ will allow one to recover all $r$ columns of $\mA^{(1)}$ and $\mA^{(2)}$. Furthermore, any mode-$(1,2,k)$ subtensor has CP-rank $r$. So, after performing Jennrich's algorithm to learn the factor matrices $\mA^{(1)}$ and $\mA^{(2)}$, one can perform the slice-by-slice censored least squares procedure (below) on any mode-$(1,2,k)$ subtensor to learn the mode-$k$ factor matrix $\mA^{(k)}$. 

\begin{algorithm}[H]
\textbf{Inputs:} 
\\
Subset $\Omega \subset [n_1] \times [n_2] \times [n_3]$ of observed entries 
\\
Samples $\calX_{i_1,i_2,i_3}$ for $(i_1,i_2,i_3) \in \Omega$ of an order-$3$ tensor $\calX \in \R^{n_1 \times n_2 \times n_3}$ with CP-rank $k \le r$. 
\\
Factor matrices $\mA^{(1)} \in \R^{n_1 \times r}$ and $\mA^{(2)} \in \R^{n_2 \times r}$ such that $\calX = \left[\left[\mA^{(1)},\mA^{(2)},\mA^{(3)}\right]\right]$ for some $\mA^{(3)} \in \R^{n_3 \times r}$
\begin{algorithmic}[1]
\caption{Slice-by-slice Censored Least Squares}
\label{alg:CLS}
    \For{$i_3 \in [n_3]$}
    \State Solve the system of equations for $\{\mA^{(3)}_{i_3,\ell}\}_{\ell = 1}^{r}$ \vspace{-0.1 in} $$\sum_{\ell = 1}^{r}\mA^{(1)}_{i_1,\ell}\mA^{(2)}_{i_2,\ell}\mA^{(3)}_{i_3,\ell} = \calX_{i_1,i_2,i_3} \quad \text{for} \quad (i_1,i_2) \quad \text{s.t.} \quad (i_1,i_2,i_3) \in \Omega$$ \vspace{-0.1 in}
    \EndFor
    \State \Return $\mA^{(3)}$
\end{algorithmic}
\end{algorithm}

We now state simplified version of our adaptive Tensor Deli algorithm for order-$d$ tensors with no zeros in $\mA^{(3)},\ldots,\mA^{(d)}$ as well as a theorem with the theoretical guarantees.

\begin{algorithm}[H]
\begin{algorithmic}[1]
\caption{Adaptive Tensor Deli for Order-$d$ tensors with no zeros in $\mA^{(3)},\ldots,\mA^{(d)}$}
\label{alg:zerofree_adapt}
    \State Pick any subset $S \subset [n_3]$ with $|S| = s$ indices
    \State Pick any indices $i^*_3 \in [n_3], \ldots, i^*_d \in [n_d]$.
    \For{$i_3 \in S$}
        \State Use the algorithm in \cite{BalcanZhang2016} to sample and complete $\calT_{:,:,i_3,i^*_4,\ldots,i^*_d}$.
    \EndFor
    \State $[\mA^{(1)},\mA^{(2)}] = \text{Jennrich}(\calT_{:,:,S,i^*_4,\ldots,i^*_d})$ 
    \State Perform QR decomposition with pivoting on $(\mA^{(1)} \odot \mA^{(2)})^T$ to identify a subset of $r$ linearly independent rows of $\mA^{(1)} \odot \mA^{(2)}$. Let $L \subset [n_1] \times [n_2]$ be the subset of indices $(i_1,i_2)$ corresponding to these rows.
    \For{$k = 3,\ldots,d$}
        \State Sample $\calT_{i_1,i_2,i^*_3,\ldots,i^*_{k-1},i_k,i^*_{k+1},\ldots,i^*_d}$ for $(i_1,i_2) \in L$ and $i_k \in [n_k]$.
        \State $\mA^{(k)} = \text{CensoredLeastSquares}(\calT_{:,:,i^*_3,\ldots,i^*_{k-1},:,i^*_{k+1},\ldots,i^*_d}, S' \times [n_k],\mA^{(1)},\mA^{(2)})$
    \EndFor
    \For{$\ell = 1,\ldots,r$}
        \State $\mA^{(d)}_{:,\ell} \leftarrow \left(\prod_{k = 3}^{d-1}\mA^{(k)}_{i^*_k,\ell}\right)^{-1}\mA^{(d)}_{:,\ell}$
    \EndFor
    \State \Return $\mA^{(1)},\ldots,\mA^{(d)}$
\end{algorithmic}
\end{algorithm}

\begin{theorem}
\label{thm:zerofree_adapt}
Suppose that the factor matrices satisfy the following assumptions


\begin{itemize}
    \item[(a)] $\mu(\mA^{(1)}) \le \mu_0$
    \item[(b)] $\rank(\mA^{(1)}) = \rank(\mA^{(2)}) = r$
    \item[(c)] every $s \times r$ submatrix of $\mA^{(3)}$ has Kruskal rank $\ge 2$
    \item[(d)] for $k = 3,\ldots,d$, $\mA^{(k)}$ has no entries which are exactly $0$.
\end{itemize}

Then, with probability at least $1-s\delta$, Algorithm~\ref{alg:zerofree_adapt} completes $\calT$ and uses at most \[O(s\mu_0 n_2r\log(r/\delta)) + sn_1r + r\sum_{k = 3}^{d}n_k\] samples.
\end{theorem}

For a typical order-$d$ tensor of size $n \times \cdots \times n$, we have $s = 2$ and Algorithm~\ref{alg:zerofree_adapt} completes $\calT$ with probability at least $1-2\delta$ using $O(\mu_0 nr\log(r/\delta))+dnr$ samples. By comparison, the result in \cite{krishnamurthy2013low} requires $\mathcal{O}(\mu_0^{d-1} nr^{d-1/2}\log(dr/\delta))$ samples. The sample complexity of our algorithm has a more favorable dependence on the coherence $\mu_0$ and rank $r$. Our result only requires coherence assumptions about $\mA^{(1)}$, instead of $\mA^{(1)},\ldots,\mA^{(d-1)}$, but at the expense of requiring assumptions (b) and (c) for Jennrich's algorithm to work and assumption (d) for every slice to contain a non-zero amount of each rank one component. Additionally, Algorithm~\ref{alg:zerofree_adapt} runs in $O(n^2r^2 + (d-2)nr^2)$ operations. For a detailed discussion on the runtime of Algorithm~\ref{alg:zerofree_adapt} (as well as its generalization where we relax the assumption that $\mA^{(3)},\ldots,\mA^{(d)}$ have no zeros), see Appendix~\ref{sec:AdaptiveRuntime}.

We also state simplified version of our nonadaptive Tensor Deli algorithm for order-$d$ tensors with no zeros in $\mA^{(3)},\ldots,\mA^{(d)}$ as well as a theorem with the theoretical guarantees.

\begin{algorithm}[H]
\begin{algorithmic}[1]
\caption{Nonadaptive Tensor Deli for Order-$d$ tensors with no zeros in $\mA^{(3)},\ldots,\mA^{(d)}$}
\label{alg:zerofree_nonadapt}
    \State Pick any subset $S \subset [n_3]$ with $|S| = s$ indices
    \State Pick any indices $i^*_3 \in [n_3], \ldots, i^*_d \in [n_d]$.
    \State Generate a random subset of sample locations $\Omega_M \subset [n_1] \times [n_2] \times S \times \{i^*_4\} \times \cdots \times \{i^*_d\}$ by independently including each entry with probability $c_0 \dfrac{\mu_0 r\log^2(n_1+n_2)}{\min\{n_1,n_2\}}$.
    \For{$k = 3,\ldots,d$}
        \State Generate a random subset of sample locations $\Omega_k \subset [n_1] \times [n_2] \times \{i^*_3\} \times \cdots \times \{i^*_{k-1}\} \times [n_k] \times \{i^*_{k+1}\} \times \cdots \times \{i^*_d\}$ by independently including each entry with probability $c_3\dfrac{\mu_0^2 r^2\log n_k}{n_1n_2}$.
    \EndFor
    \State Sample $\calT_{i_1,\ldots,i_d}$ for $(i_1,\ldots,i_d) \in \Omega := \Omega_M \cup \bigcup_{k = 3}^{d}\Omega_k$
    \For{$i_3 \in S$}
        \State Use nuclear norm minimization to complete $\calT_{:,:,i_3,i^*_4,\ldots,i^*_d}$ using only sample locations in $\Omega_M$
    \EndFor
    \State $[\mA^{(1)},\mA^{(2)}] = \text{Jennrich}(\calT_{:,:,S,i^*_4,\ldots,i^*_d})$ 
    \For{$k = 3,\ldots,d$}
        \State $\mA^{(k)} = \text{CensoredLeastSquares}(\calT_{:,:,i^*_3,\ldots,i^*_{k-1},:,i^*_{k+1},\ldots,i^*_d}, \Omega_t,\mA^{(1)},\mA^{(2)})$
    \EndFor
    \For{$\ell = 1,\ldots,r$}
        \State $\mA^{(d)}_{:,\ell} \leftarrow \left(\prod_{k = 3}^{d-1}\mA^{(k)}_{i^*_k,\ell}\right)^{-1}\mA^{(d)}_{:,\ell}$
    \EndFor
    \State \Return $\mA^{(1)},\ldots,\mA^{(d)}$
\end{algorithmic}
\end{algorithm}

\begin{theorem}
\label{thm:zerofree_nonadapt}
Suppose that the factor matrices satisfy the following assumptions
\begin{enumerate}
    \item[(a)] $\mu(\mA^{(1)}) \le \mu_0$ and $\mu(\mA^{(2)}) \le \mu_0$
    \item[(b)] $\rank(\mA^{(1)}) = \rank(\mA^{(2)}) = r$
    \item[(c)] every $s \times r$ submatrix of $\mA^{(3)}$ has Kruskal rank $\ge 2$
    \item[(d)] for $k = 3,\ldots,d$, $\mA^{(k)}$ has no entries which are exactly $0$.
\end{enumerate}

Then, Algorithm~\ref{alg:zerofree_nonadapt} completes $\calT$ with probability at least \[1-c_1s(n_1+n_2)^{-c_2}-r\sum_{k = 3}^{d}n_k^{-(c_3-1)}.\] Furthermore, with high probability, the number of randomly drawn samples is at most \[O\left(s\mu_0r\max\{n_1,n_2\}\log^2(n_1+n_2) + \mu_0^2r^2\sum_{k = 3}^{d}n_k\log n_k\right).\] Here, $c_0, c_1, c_2 > 0$ are the constants in \cite{chen2015incoherence}, and $c_3$ can be any number larger than $1$.
\end{theorem}

For a typical order-$d$ tensor of size $n \times \cdots \times n$, we have $s = 2$ and Algorithm~\ref{alg:zerofree_adapt} completes $\calT$ with high probability using $O(\mu_0 nr\log^2 n + (d-2)\mu_0^2nr^2\log n)$ samples. By varying the sampling probability over the tensor, we are able to reduce the sample complexity from having an $n^{d/2}$ dependence on $n$ to having a slightly worse than linear dependence on $n$. Additionally, Algorithm~\ref{alg:zerofree_nonadapt} runs in polynomial time. For a detailed discussion on the runtime of Algorithm~\ref{alg:zerofree_nonadapt} (as well as its generalization where we relax the assumption that $\mA^{(3)},\ldots,\mA^{(d)}$ have no zeros), see Appendix~\ref{sec:NonadaptiveRuntime}.

\section{General Results}
\label{sec:GeneralResults}
Now, we will consider the general case where the tensor has order-$d$, and for $k = 3,\ldots,d$, the factor matrix $\mA^{(k)}$ is allowed to have up to $zn_k$ zeros per column where $z \in [0,1)$. Essentially, our algorithms aim to complete $m$ mode-$(1,2,3)$ subtensors of size $n_1 \times n_2 \times s$ and perform Jennrich's algorithm on each one until it learns all $r$ columns of $\mA^{(1)}$ and $\mA^{(2)}$. Then, for each $k = 3,\ldots,d$, it will perform slice-by-slice censored least squares on $m$ mode-$(1,2,k)$ subtensors to learn all $r$ columns of $\mA^{(k)}$. The parameter $m$ needs to be chosen large enough to ensure that all $r$ rank-$1$ components can be learned.

\subsection{Adaptive Tensor Sandwich}

\begin{theorem}
\label{thm:adaptive_sampling}
Suppose that the factor matrices satisfy the following assumptions

\begin{enumerate}
    \item[(a)] $\mu(\mA^{(1)}) \le \mu_0$
    \item[(b)] $\rank(\mA^{(1)}) = \rank(\mA^{(2)}) = r$
    \item[(c)] every $s \times r$ submatrix of $\mA^{(3)}$ has Kruskal rank $\ge 2$
    \item[(d)] for $k = 3,\ldots,d$, each column of $\mA^{(k)}$ has at most $zn_k$ entries which are $0$, where $z \in [0,1)$.
\end{enumerate}

Then, for any positive integer $m$, Algorithm~\ref{alg:adaptive_sampling} (see Appendix~\ref{sec:GeneralAlgorithmAdaptive}) both completes $\calT$ and uses at most \[O\left(sm\mu_0 n_2r\log(r/\delta)\right) + smn_1r + mr\sum_{k = 3}^{d}n_k\] samples with probability at least $1-sm\delta-(d-2)r(1-(1-z)^{d-3})^m$.
\end{theorem}

\subsection{Nonadaptive Independent Sampling Tensor Sandwich}

\begin{theorem}
\label{thm:non_adaptive_dust_sampling}
Suppose that the factor matrices satisfy the following assumptions

\begin{enumerate}
    \item[(a)] $\mu(\mA^{(1)}) \le \mu_0$ and $\mu(\mA^{(2)}) \le \mu_0$
    \item[(b)] $\rank(\mA^{(1)}) = \rank(\mA^{(2)}) = r$
    \item[(c)] every $s \times r$ submatrix of $\mA^{(3)}$ has Kruskal rank $\ge 2$
    \item[(d)] for $k = 3,\ldots,d$, each column of $\mA^{(k)}$ has at most $zn_k$ entries which are $0$, where $z \in [0,1)$.
\end{enumerate}

Then, for any positive integer $m$, Algorithm~\ref{alg:non_adaptive_dust_sampling} (see Appendix~\ref{sec:GeneralAlgorithmNonadaptive}) completes $\calT$ with probability at least \[1-c_1sm(n_1+n_2)^{-c_2}-(d-2)r(1-(1-z)^{d-3})^m - mr\sum_{k = 3}^{d}n_k^{-(c_3-1)}.\] Furthermore, with high probability, the number of randomly drawn samples is at most \[O\left(sm\mu_0 \max\{n_1,n_2\}r\log^2(n_1+n_2) + m\mu_0^2r^2\sum_{k = 3}^{d}n_k \log n_k\right).\] Here, $c_0, c_1, c_2 > 0$ are the constants in \cite{chen2015incoherence}, and $c_3$ can be any number larger than $1$.
\end{theorem}

Remark: Choosing $m \approx \dfrac{\log\left(\tfrac{(d-2)r}{\delta'}\right)}{-\log(1-(1-z)^{d-3})}$ for some small $\delta' > 0$ will keep both the failure probability and the total number of samples reasonably small.

\section{Numerical Experiments}
\label{sec:Experiments}
In this section, we show that Tensor Deli (TD) can complete tensors of simulated data as well as real-world data using both adaptive and non-adaptive sampling. Once sample complexity bounds are satisfied in either setting, TD can achieve low levels of relative error for completing tensors with good low CP-rank approximations by looking at only a small fraction of their total entries. We also demonstrate empirically that TD works for four mode tensors. In prior works such as \cite{krishnamurthy2013low}, results of this kind are given only for two-mode tensors; in \cite{liu2020tensor} experiments are only run on three-mode tensors. Furthermore, we show how using Tensor Deli to initialize a masked-alternating least squares (ALS) scheme, with even a few iterations, provides a significant improvement in completion accuracy in both simulated and real world data. We further compare adaptive TD results to \cite{krishnamurthy2013low}, and show that Tensor Deli can perform well in terms of accuracy and sample complexity while also providing a factorized tensor, and while being significantly faster in terms of runtime.\footnote{All data and code for this section are available at \url{https://github.com/cahaselby/TensorDeli}.}

\subsection{Simulated Data}

Here we present empirical results on simulated data. In all cases, the data is generated by drawing factor matrices $\mA^{(1)}, \dots, \mA^{(d)} \in \mathbb{R}^{n\times r}$ with i.i.d standard Gaussian entries and then $\ell^2$-normalizing the columns. For the simulated data the side-lengths $n$ are the same for each of the $d$ modes, and thus the tensor has $n^d$ total entries. We then have the option to weight the components using decaying weights described by a parameter $\alpha>0$, i.e. 
\begin{equation} \label{eqn:data_gen}
\calT = \sum_{\ell = 1}^{r}  \left(\frac{1}{\ell^\alpha}\right)  \va^{(1)}_{\ell} \circ \cdots \circ \va^{(d)}_{\ell}.
\end{equation}
In each experiment, the median of the errors is taken over ten independent trials. In each Tensor Deli experiment frontal slices are first selected to complete, and within these slices we sample using a budget of $m=\gamma n^2$ samples (adaptively or non-adaptively), where $\gamma\in [0,1]$ is the proportion of the total $n^2$ entries of the slice available for sampling in these first matrix completion steps. Slices are then completed using a semi-definite programming formulation of nuclear norm minimization solved via Douglas-Rachford splitting, see \cite{ocpb:16}. In this step for the case of solving the semi-definite program, convergence is declared once the primal residual, dual residual, and duality gap are all below $10^{-8}$ for each of the $s$ selected slices, or else after 10,000 iterations per slice, whichever comes first. We note that the accuracy of this matrix completion step influences numerically what is achievable in terms of overall accuracy for the completed tensor, and that the error for completing these slices is compounded in the subsequent steps (even in the absence of noise). This explains the apparent ``leveling off'' of the relative error even as sample complexity or signal-to-noise ratios increase. The number of fibers to sample which are used to estimate the remaining $d-2$ factor matrices is given by $\delta r$, where $\delta$ is an oversampling factor of size at least one.

In Figure~\ref{fig:basic_study_adapt} we show for $d=4$, $n=100$, $\alpha = 2$ and for three different ranks $r=5,10,15$ non-adaptive TD is able to reach relative errors at or lower than 1\% using less than one tenth of one percent of the total tensor entries. In Figure~\ref{fig:basic_study_adapt}, we show the median relative error for Tensor Deli and the relative error after ten iterations of masked-ALS is performed on the result, i.e. the second set of experiments depend on the first in that masked-ALS is initialized using the Tensor Deli estimates of the CP factorization, and using the same revealed entries. In this set of experiments, we perform the matrix completion step of our method using the sampling technique in \cite{Ward2015}, and nuclear norm minimization.  As is well documented now, ALS can be quite sensitive to initialization and prone to ``swamps'' where accuracy stagnates over a large number of iterations. This figure shows however that Tensor Deli can provide high quality initializations, and with only a small number of iterations we can improve the accuracy of the estimate of the completed tensor by an order of magnitude without even needing to reveal more entries. Additionally, we have ALS alone applied to the same tensor using the same revealed entries, initialized using the masked SVDs of the unfoldings of the tensor. Empirically we have observed uniform sampling at the same complexity performs the same for ALS, and in either case of sampling pattern, masked ALS alone is unable to achieve comparable relative errors.

\begin{figure}
\centering
\includegraphics[scale=0.5]{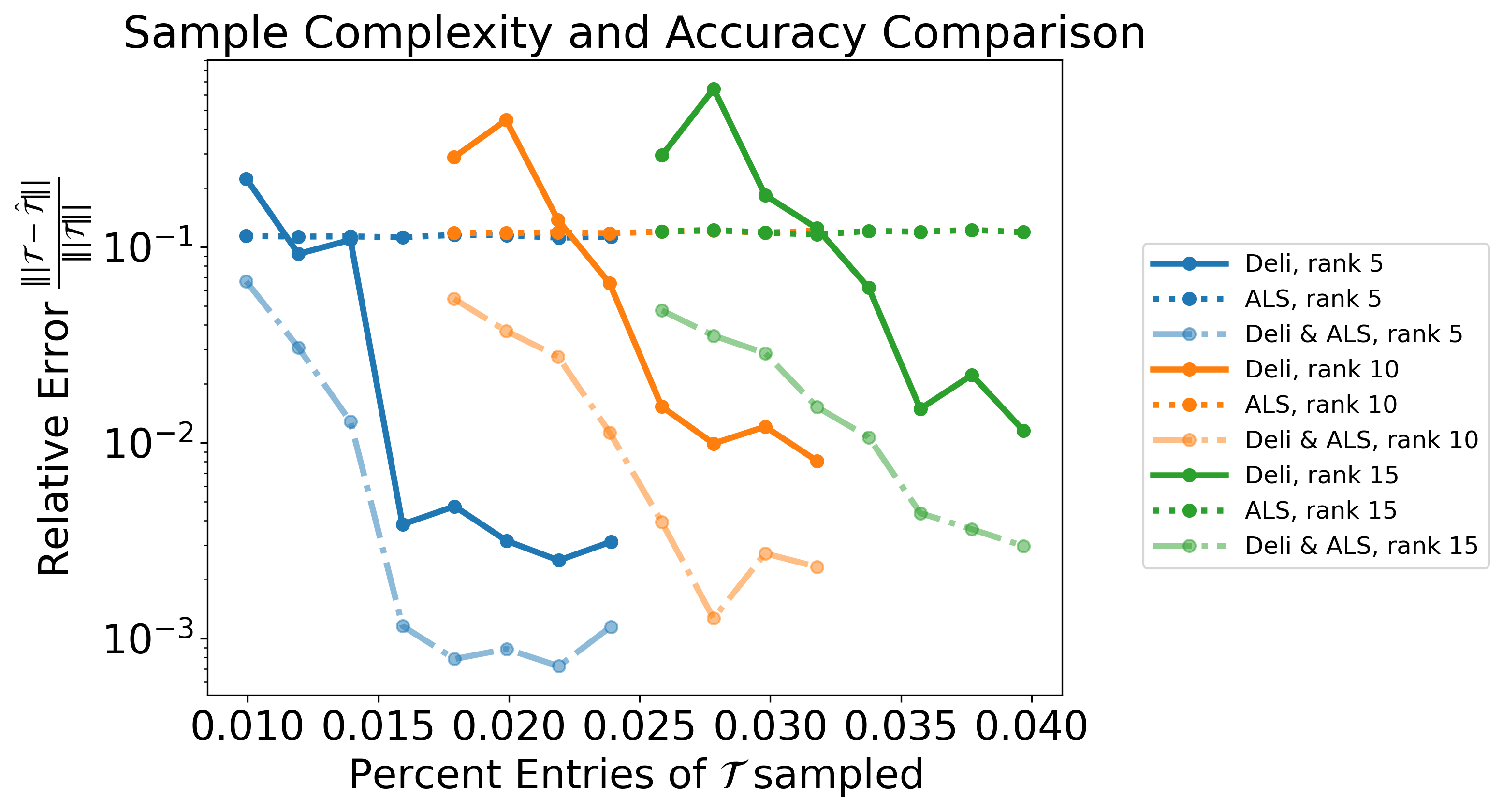}
\caption{Median relative error (log-scaled) of completed four mode tensors of varying rank as sample complexity increases without noise. Each value is the median of ten trials, $n=100$, $d=4$, $s=2, \gamma \in [0.1,0.8], \delta=8$. We compare the relative errors of adaptive Tensor Deli before and after ten iterations of masked-ALS, as well as just masked ALS alone.}
\label{fig:basic_study_adapt}
\end{figure}

In Figure~\ref{fig:basic_study_nonadapt} we show the errors for adaptive TD, as in Algorithm~\ref{alg:adaptive_sampling}, versus non-adaptive, as in Algorithm~\ref{alg:non_adaptive_dust_sampling}, strategy for each rank with the same experimental setup as in Figure~\ref{fig:basic_study_adapt}. Not surprisingly, in the case of random low-rank tensors, the density of the sampling in the slices drives the accuracy, and adapting the sampling pattern appears to have little effect - the slices themselves are already highly incoherent, so there is little to gain by adaptive sampling. As we discuss in the next section, on real world data with more structure, adaptive sampling does allow for more significant improvements.  

\begin{figure}
\centering
\includegraphics[scale=0.5]{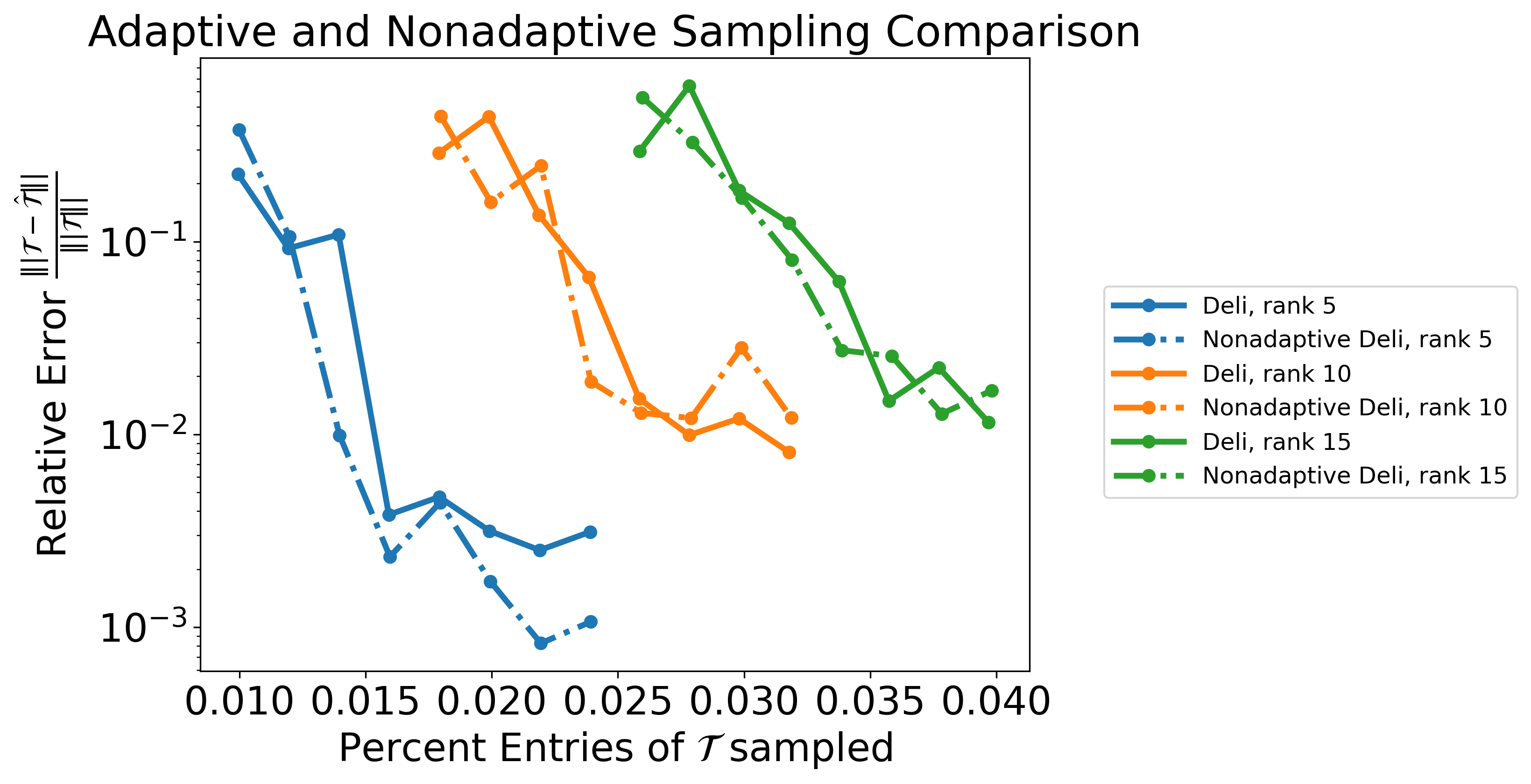}
\caption{Median relative error (log-scaled) of completed four mode tensors of varying rank as sample complexity increases without noise. Each value is the median of ten trials, $n=100$, $d=4$, $s=2,\gamma \in [0.1,0.8], \delta=8$. We show errors for both adaptive and non-adaptive TD sampling schemes.}
\label{fig:basic_study_nonadapt}
\end{figure}

In Figures~\ref{fig:ks_compare} and \ref{fig:ks_noise} we show a comparison of the method proposed in \cite{krishnamurthy2013low} (denoted KS in the figure) and adaptive Tensor Deli, with ten iterations of ALS. For our accuracy comparison in plot (a) of Figure~\ref{fig:ks_compare}, we compare the median relative error adaptive TD with KS. Plot (b) shows the mean runtime, and plot (c) has the mean sample complexity. In each plot it is the same ten, independent trials, that are averaged for each choice of side-length of three mode tensors from $n=50$ to $n=120$. In each experiment, the rank is fixed to be a tenth of the side-length, $r=0.1n$. In order to best control for the implementation of the matrix completion phase in the comparisons, Tensor Deli uses the KS algorithm to complete the densely sampled slices in these figures. This demonstrates that Tensor Deli can easily accommodate the use of other matrix completion methods, and potentially with a substantial benefit to the overall runtime, sample complexity, and accuracy. The advantages from Tensor Deli come from the fact that samples and runtime spent on completing a smaller subtensor accurately, and obtaining estimates for some of the factors, payoff later when completing the rest of the tensor. Another noteworthy phenomenon of practical importance for the KS algorithm is that, while the error does appear monotonic in terms of the fiber sampling parameter, which is the parameter that drives the overall sample complexity and runtime, this does not necessarily mean that the error is monotonic in terms of overall number of entries revealed. This is because it can be inefficient to under sample the number of entries in a fiber (or in any sub-tensor in the recursion)  when testing if fibers are in the span of the learned basis. These false negatives will result in more fibers being fully revealed which do not meaningfully expand the basis. Additionally, KS does not alone output the CP factors of the tensor it completes, which may themselves be what's of interest.

Next we compare the performance of these two adaptive algorithms in the presence of noise. In Figure~\ref{fig:ks_noise} we have a similar setup but now we fix $n=50$, $d=3$, $\alpha=1$ and vary for ranks $r=4,6,8$. Sampling parameter for fibers and faces is fixed at $0.7$ respectively for KS and Tensor Deli, and we add mean-zero i.i.d. Gaussian noise to each entry in our tensor. For each trial, the noise tensor $\mathcal{N}$ is scaled to the appropriate signal-to-noise ratio along the horizontal axis, i.e. $\text{SNR} = 20\log_{10}\tfrac{\|\calT\|}{\|\mathcal{N}\|}$. This shows that Tensor Deli performance is comparable to (or better than) KS depending on the amount of noise - however, it does this at significantly smaller proportion of entries being sampled. In Figure~\ref{fig:ks_noise}, KS ranges from 12.9\% to 80.2\% of total number of entries sampled depending on the level of noise, whereas the highest proportion sampled for Tensor Deli is 6.2\%.

\begin{figure}
\centering
\includegraphics[width=\columnwidth]{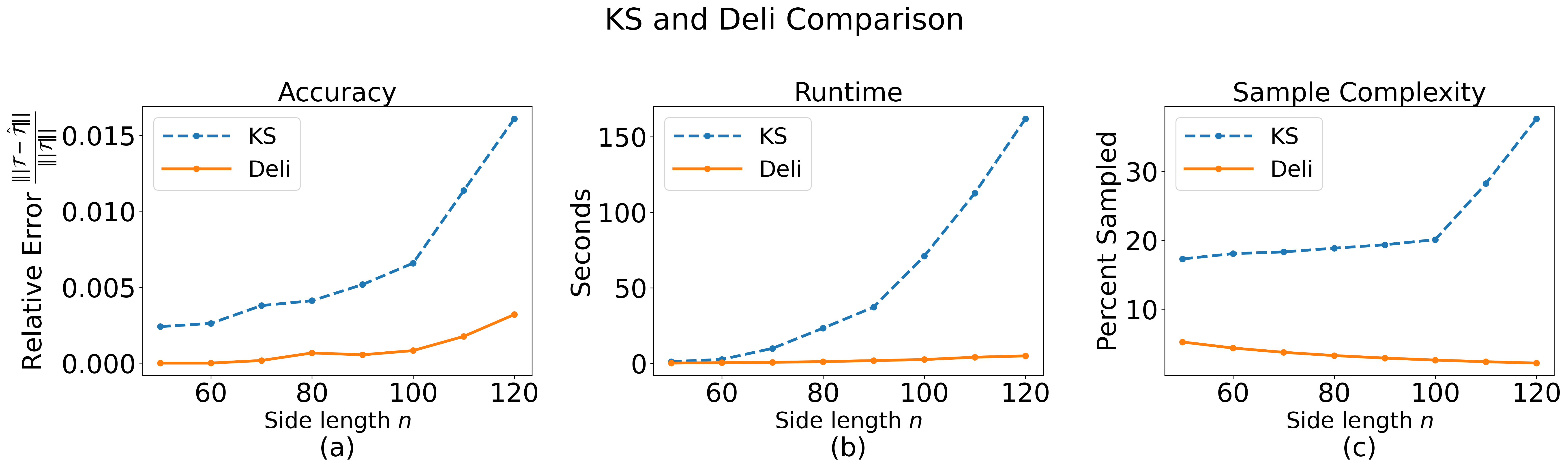}
\caption{Comparison of runtime, accuracy and sample complexity for different sized three mode tensors, $n\in[50,120]$, $d=3$, $r=0.1n$. Tensor Deli utilizes KS to complete two densely sampled slices and also performs ten iterations of ALS. This is compared with KS alone used to complete the entire tensor. For this simulated data, Tensor Deli is able to achieve faster runtimes, better accuracy, and at a lower sample complexity than KS. Here we average errors, runtime, and utilized samples over ten independent trials for each choice of $n$.}
\label{fig:ks_compare}
\end{figure}

\begin{figure}
\centering
\includegraphics[scale=0.4]{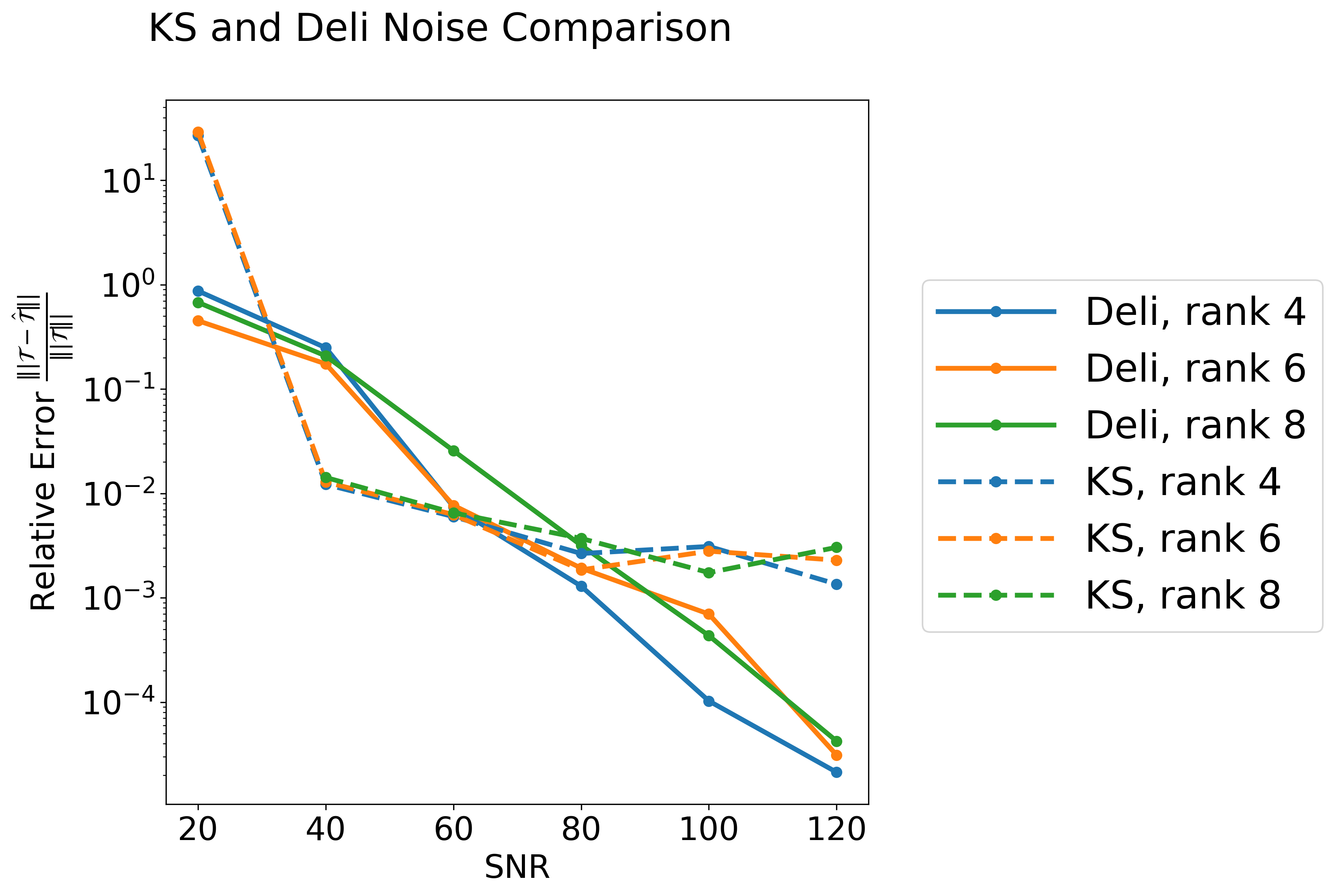}
\caption{Median relative error (log-scaled) of completed three mode tensors of varying rank with additive noise for adaptive Tensor Deli with adaptive sampling method, and KS. Each value is the median of ten trials, for Tensor Deli $n=50$, $d=3$, $s=2, \delta=8$, for KS the proportion of entries that can be sampled for a fiber is 0.7.}
\label{fig:ks_noise}
\end{figure}

\subsection{Real Data from Applications}

We also apply Tensor Deli on data-sets in the following application areas: chemo-metrics, and hyper-spectral imaging. The chemo-metrics data-set is as used in \cite{fluordata2005}. It is fluorensce measured from known analytes intended for calibration purposes. There are a total of 405 fluorophores of six different types. For each sample an Excitation Emission Matrix (EEM) was measured - i.e. fluorescence intensity is measured for different set levels of excitation wavelength and emission wavelength. The EEM of an unknown sample can be used for its identification as well as for studying other properties of interest for a given analyte. In this data-set there are 136 emission wavelengths and 19 excitation wavelengths. In the original tensor, there are five replicates per sample, however in our experiment we discarded all but the first set of replicates. As a result, the original tensor is a three mode tensor of size $405 \times 136 \times 19$. We compare the completed tensor for five different choices of rank, $r=11,15,19$ to the original data. We set $s=4$, and thus complete four frontal slices. We fix frontal slices where the third mode is at the equally-spaced middle indices $5,9,13,17$ across all experiments in order to facilitate better comparison. In Figure~\ref{fig:big_fluor_rank}, we show the recovered tensor for a representative frontal slice at each of the different ranks. Below this in the same figure, we have fixed a representative lateral slice, which corresponds to a completed EEM for a particular sample of an analyte at the different ranks. In all cases, the total number of entries revealed is between 10 and 11\% of the tensor's entries, we performed adaptive sampling and ten iterations of masked-ALS on the resulting CP estimate using the same revealed samples as for TD, and the sampling parameters $\gamma$ and $\delta$ are set to $0.5$ and $10$, respectively.

Furthermore, in Figure~\ref{fig:fluor_adapt} we show evidence that indeed, in practical applications, there is quite clearly a benefit to sampling adaptively in terms of the trade off between accuracy and sample complexity. In this figure, for a fixed rank of $15$ and a fixed sample complexity using four slices and the sampling parameters of $0.5$ and $10$ for $\gamma$ and $\delta$, using the same lateral slice as before in Figure~\ref{fig:big_fluor_rank}, we show the completed EEM and corresponding relative error for this slice using the adaptive Tensor Deli (Algorithm~\ref{alg:adaptive_sampling}), the adaptive Tensor Deli with ten iterations of masked-ALS, and non-adaptive Tensor Deli (Algorithm~\ref{alg:non_adaptive_dust_sampling}). In this case, the additional ALS iterations show only modest improvement to the overall estimates, however the adaptive scheme's accuracy has a much larger effect. This is likely do to the fact the coherence of the data and other parameters related to the problem are far from ideal, especially when compared to the simulated case from earlier.  

\begin{figure}
\centering
\includegraphics[scale = 0.32]{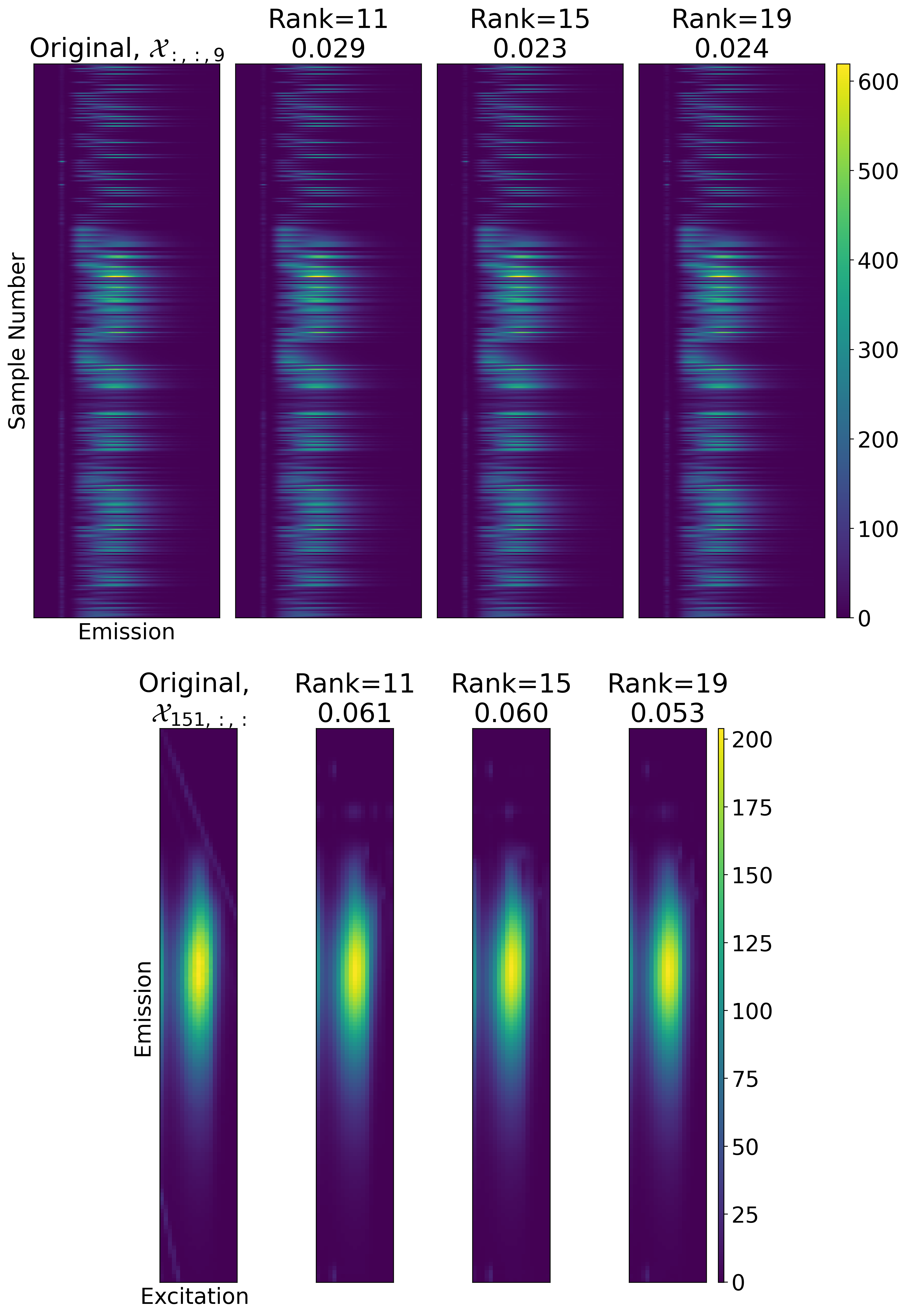}
\caption{Slices of a completed tensor for fluorence data. The top row shows the original frontal slice $\mathcal{X}_{:,:,9}$ and the same slice completed at ranks $11,15,17$ with their corresponding relative errors. The bottom row shows a lateral slice $\mathcal{X}_{151,:,:}$, which corresponds then to the EEM for the single analyte number 151 in the dataset.}
\label{fig:big_fluor_rank}
\end{figure}

\begin{figure}
\centering
\includegraphics[scale = 0.35]{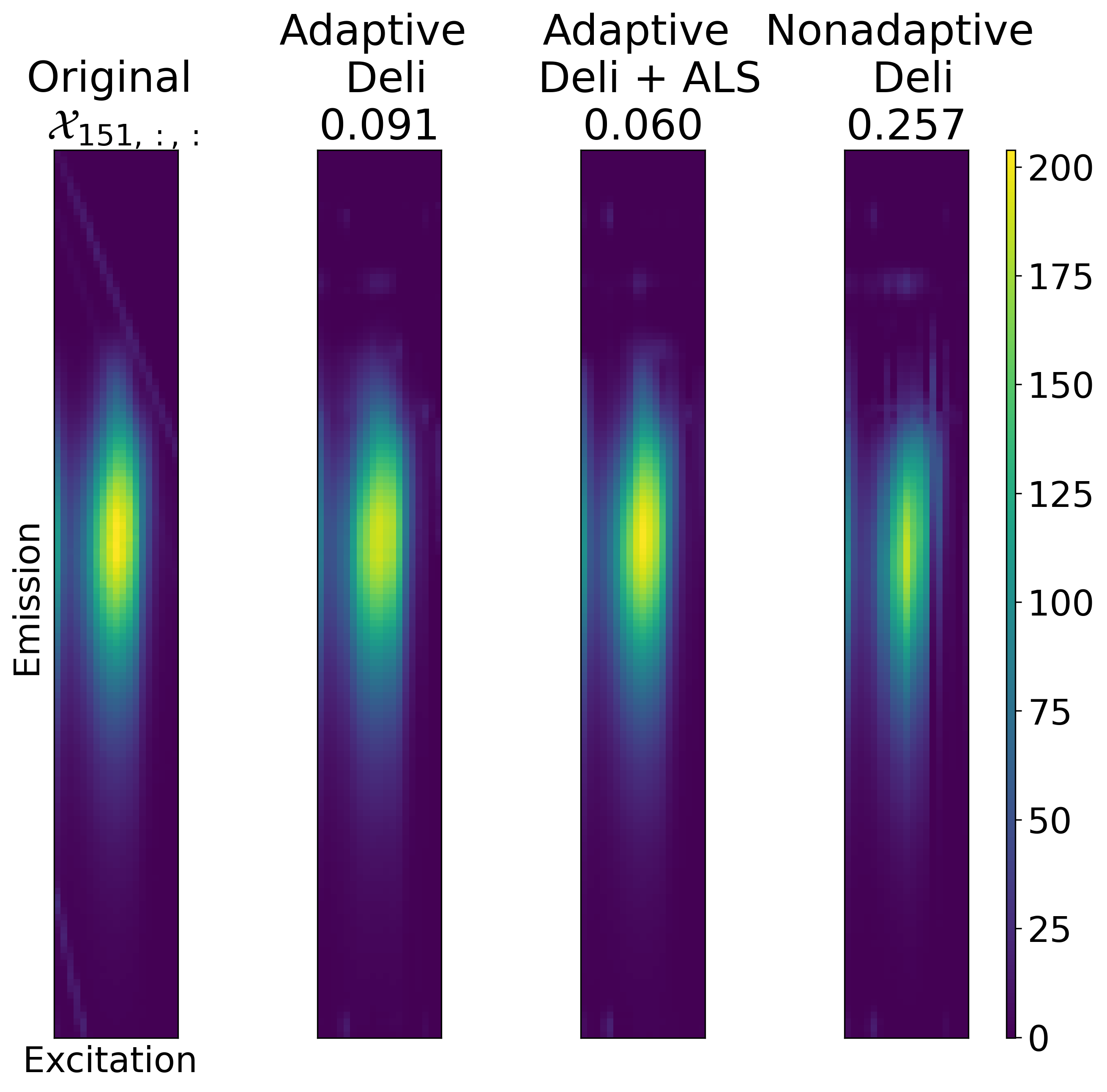}
\caption{Lateral slice corresponding to the EEM for analyte number 151 in the fluorence data is compared with the corresponding slice completed from adaptive, adaptive with extra iterations of ALS, and non-adaptive Tensor Deli. Relative error for the slice is listed for each method. Rank is fixed at $15$ for each, sample complexity is fixed at about 10\% of entries using sampling parameters of $0.5$ and $10$ for $\gamma$ and $\delta$ respectively.}
\label{fig:fluor_adapt}
\end{figure}

The next application is hyperspectral imaging. The data is as found in \cite{pinesdata2015}. The hyperspectral sensor data was acquired in June 1992 and consists of aerial images of an approximately two mile by to mile Purdue University Agronomy farm (originally intended for soils research). The data consists of 200 images at different wave lengths that are 145 by 145 pixels each. We thus form a three mode, $145 \times 145 \times 200$ data tensor. We complete the tensor using $s=9$, where we have fixed these frontal slices at indices $20,40,60,80,100,120,140,160,180$ across all experiments to facilitate comparison. Shown in Figure~\ref{fig:pine_study}, we have completed the tensor for ranks $r=30,40,50,60$ and displayed a fixed, representative frontal slice at index 48. The first row consists of data completed using Tensor Deli with adaptive sampling, with parameters $\gamma=0.7$ and $\delta=10$ for a total sample budget that is about 4.5\% of the total entries. The second row then performs ten iterations of masked-ALS using the initialization of the top row and the same revealed entries. The last row consists of ten iterations masked-ALS, which is initialized using the SVD of the flattenings, e.g. see \cite{Tomasi2005PARAFACAM}, and uses the same revealed entries as the other experiments at this rank.

Empirically, we observe the hyperspectral image dataset is imperfectly approximated by a low-rank CP decomposition to begin with. In Figure~\ref{fig:pine_study}, we see Tensor Deli alone performs in terms of global relative error certainly no better than masked-ALS alone on this data. However, it provides a superior initialization to ALS, as we can see in the second row. Moreover, we observe there is a qualitative difference in the completed slices and their errors. In the ALS alone case, the error appears achieved by a sort of local averaging (i.e., blurring) of the intensity of the pixels, whereas Tensor Deli does capture contrasts and local features more distinctly, and the errors are largest on rows and columns it ``misses'' or imperfectly completes. This can be seen by looking at a heat map of the error slice by slice. Looking at the middle row, we can see there is a distinct advantage in combining these two types of methods to achieve the best sort of completion for this dataset. 

\begin{figure}
\centering
\includegraphics[scale=0.375]{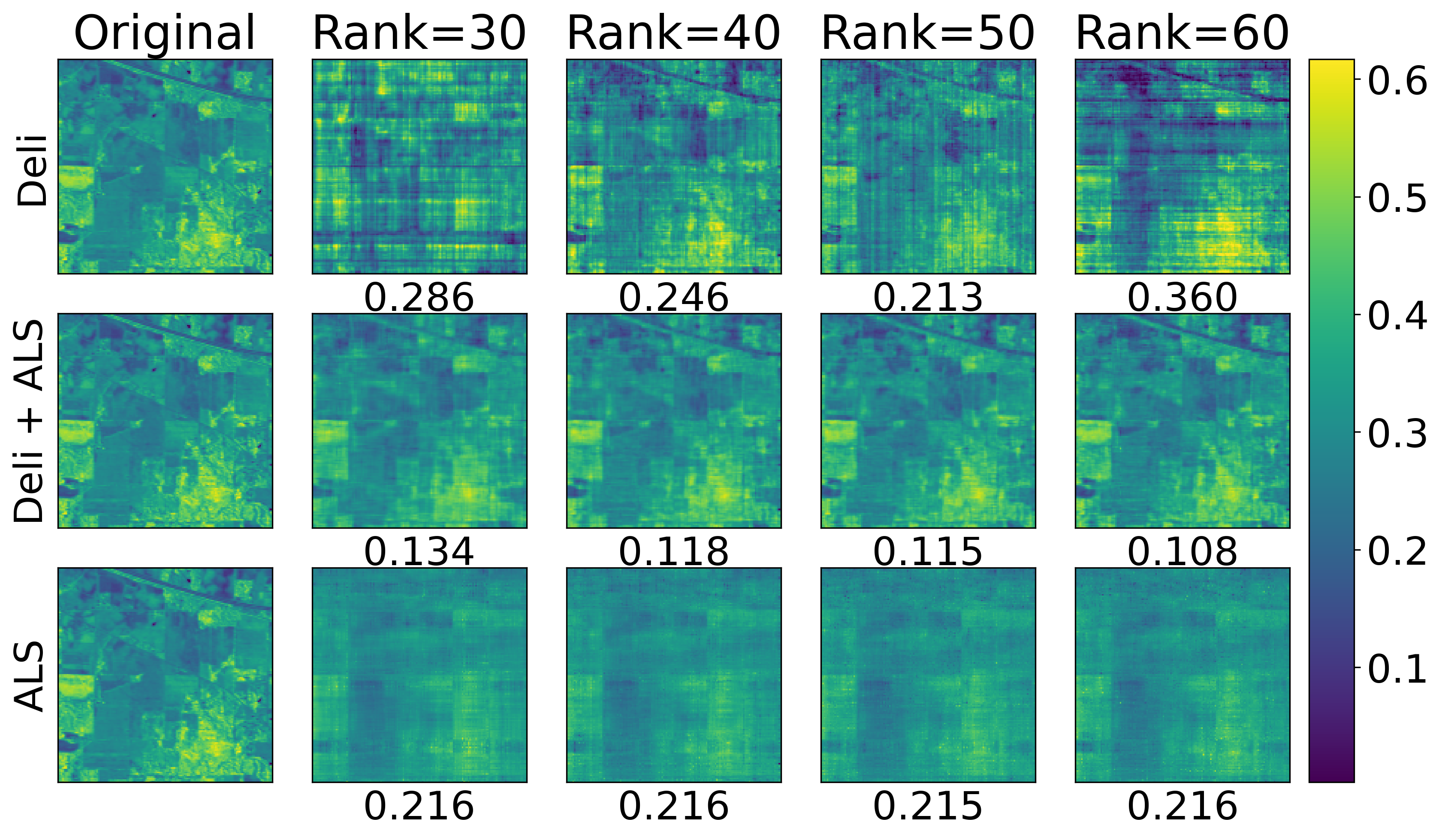}
\caption{Slices of a completed tensor for hyperspectral data. The first column shows the original frontal slice $\mathcal{X}_{:,:,48}$ and the same slice completed at ranks $30,40,50,60$ with the corresponding relative error below for three different methods. The top row is adaptive Tensor Deli with sampling parameters $\gamma=0.7$, $\delta=10$, the middle row shows the slice after ten iterations of ALS using the Deli initialization and the last row shows SVD initialized masked-ALS using the same revealed entries as the experiments above it.}
\label{fig:pine_study}
\end{figure}

\appendix

\section{Full Algorithms with tolerance for rank-deficient subtensors}

\subsection{Tensor Sandwich: Adaptive}
\label{sec:GeneralAlgorithmAdaptive}
\begin{algorithm}[H]
\begin{algorithmic}[1]
\caption{Adaptive Tensor Deli for Order-$d$ tensors}
\label{alg:adaptive_sampling}
    \State Pick any subset $S \subset [n_3]$ with $|S| = s$ indices
    \For{k = 3,\ldots,d}
        \State Generate a subset $Z_k \subset [n_3] \times \cdots \times [n_{k-1}] \times [n_{k+1}] \times \cdots \times [n_d]$ with $|Z_k| = m$ elements chosen uniformly at random without replacement.
    \EndFor
    \For{$(i_4,\ldots,i_d) \in Z_3$}
        \For{$i_3 \in S$}
            \State Use algorithm in \cite{BalcanZhang2016} to sample and complete $\calT_{:,:,i_3,i_4,\ldots,i_d}$.
        \EndFor
        \State Use Jennrich's algorithm on each completed subtensor $\calT_{:, :, S, i_4, \ldots, i_d}$ to recover the factor pairs $\mA^{(1)}_{:,\ell} \circ \mA^{(2)}_{:,\ell}$ for $\ell$ such that $\prod_{k = 4}^{d}\mA^{(k)}_{i_k,\ell} \neq 0$.
    \EndFor
    \State Perform QR decomposition with pivoting on $(\mA^{(1)} \odot \mA^{(2)})^T$ to identify a subset of $r$ linearly independent rows of $\mA^{(1)} \odot \mA^{(2)}$. Let $L \subset [n_1] \times [n_2]$ be the subset of indices $(i_1,i_2)$ corresponding to these rows.
    \For{$k = 3,\ldots,d$}
        \State $R = [r]$
        \For{$(i_3,\ldots,i_{k-1},i_{k+1},\ldots,i_d) \in Z_k$}
            \For{$i_k \in [n_k]$}
                \For{$(i_1,i_2) \in L$}
                    \State Sample $\calT_{i_1,i_2,i_3,\ldots,i_d}$
                \EndFor
            \EndFor   
            \State $\mB = \text{CensoredLeastSquares}(\calT_{:,:,i_3,\ldots,i_{k-1},:,i_{k+1},\ldots,i_d},L \times [n_k],\mA^{(1)},\mA^{(2)})$
            \If{$k \neq d$}
                \State $\mA^{(k)}_{:,R} = \mB_{:,R}$
            \Else
                \State $\mA^{(k)}_{:,R} = \left(\prod_{t = 3}^{d-1}\mA^{(t)}_{i_t,\ell}\right)^{-1}\mB_{:,R}$
            \EndIf
            \State $R = \{\ell : \mA^{(k)}_{:,\ell} = \vct{0}\}$
        \EndFor
    \EndFor
    \State \Return $\mA^{(1)},\ldots,\mA^{(d)}$
\end{algorithmic}
\end{algorithm}

\subsection{Proof of Theorem~\ref{thm:adaptive_sampling}}

\subsubsection{Completing $sm$ slices}
In lines 4-6 above, we use the algorithm in \cite{BalcanZhang2016} to sample and complete $|S| \cdot |Z_3| = sm$ slices of the tensor which each have dimensions $n_1 \times n_2$. Since each slice $\calT_{:,:,i_3,\ldots,i_d}$ has $\text{col-span}(\calT_{:,:,i_3,\ldots,i_d}) \subseteq \text{col-span}(\mA^{(1)})$ and $\mu(\mA^{(1)}) \le \mu_0$, each slice satisfies the assumptions of \cite{BalcanZhang2016}. Hence, for each slice $\calT_{:,:,i_3,\ldots,i_d}, i_3 \in S, (i_4,\ldots,i_d) \in Z_3$, with probability $1-\delta$, Algorithm 2 in \cite{BalcanZhang2016} uses at most $O(\mu_0 n_2r\log(r/\delta))+n_1r$ samples and completes the slice. By taking a simple union bound over all $sm$ slices, we have that with probability at least $1-sm\delta$, all $sm$ slices are completed with at most $O(sm\mu_0 n_2r\log(r/\delta))+smn_1r$ samples. For the rest of the proof, we will assume that these $sm$ slices were successfully completed with the specified number of samples.

\subsubsection{Using Jennrich's algorithm on $m$ subtensors to learn $\mA^{(1)}$ and $\mA^{(2)}$}
\label{sec:nonadaptive_proof_Jennrich}
In lines 6 and 9, we use Jennrich's algorithm \cite{moitra2018algorithmic} on $|Z_3| = m$ mode-$(1,2,3)$ subtensors of size $n_1 \times n_2 \times s$. Consider the subtensor $\calT_{:,:,S,i_4,\ldots,i_d}$ for some $(i_4,\ldots,i_d) \in Z_3$. Let \[\mC^{(3)} = \begin{bmatrix}(\mA^{(4)}_{i_4,1} \cdots \mA^{(d)}_{i_d,1})\mA^{(3)}_{:,1} & \cdots & (\mA^{(4)}_{i_4,r} \cdots \mA^{(d)}_{i_d,r})\mA^{(3)}_{:,r}\end{bmatrix},\] i.e., $\mA^{(3)}$, but with each column rescaled. Also, let $Q = \left\{\ell : \mA^{(4)}_{i_4,\ell} \cdots \mA^{(d)}_{i_d,\ell} \neq 0\right\}$. Then, we can write 

\begin{align*}
\calT_{:,:,S,i_4,\ldots,i_d} &= \sum_{\ell = 1}^{r}\mA^{(1)}_{:,\ell} \circ \mA^{(2)}_{:,\ell} \circ \mA^{(3)}_{S,\ell} \cdot (\mA^{(4)}_{i_4,\ell} \cdots \mA^{(d)}_{i_d,\ell}) 
\\
&= \sum_{\ell \in Q}\mA^{(1)}_{:,\ell} \circ \mA^{(2)}_{:,\ell} \circ \mA^{(3)}_{S,\ell} \cdot (\mA^{(4)}_{i_4,\ell} \cdots \mA^{(d)}_{i_d,\ell})
\\
&= \sum_{\ell \in Q}\mA^{(1)}_{:,\ell} \circ \mA^{(2)}_{:,\ell} \circ \mC^{(3)}_{S,\ell}
\\
&= \left[\left[\mA^{(1)}_{:,Q},\mA^{(2)}_{:,Q},\mC^{(3)}_{S,Q}\right]\right]
\end{align*}
\\
Since $\mA^{(1)}$ and $\mA^{(2)}$ have full column-rank, so do $\mA^{(1)}_{:,Q}$ and $\mA^{(2)}_{:,Q}$. Also, by assumption, $\mA^{(3)}_{S,:}$ has Kruskal-rank $\ge 2$, i.e. no column of $\mA^{(3)}_{S,:}$ is a scalar multiple of another column. Hence, no column of $\mA^{(3)}_{S,Q}$ is a scalar multiple of another column. Rescaling each column by a non-zero constant doesn't change this, so no column of $\mC^{(3)}_{S,Q}$ is a scalar multiple of another column. Since $\mA^{(1)}_{:,Q}$ and $\mA^{(2)}_{:,Q}$ have full column rank and no column of $\mC^{(3)}_{S,Q}$ is a scalar multiple of another column, the subtensor $\calT_{:,:,S,i_4,\ldots,i_d}$ satisfies the conditions for Jennrich's algorithm to recover $\mA^{(1)}_{:,Q}$ and $\mA^{(2)}_{:,Q}$, i.e. the columns $\mA^{(1)}_{:,\ell}$ and $\mA^{(2)}_{:,\ell}$ for indices $\ell$ such that $\mA^{(4)}_{i_4,\ell} \cdots \mA^{(d)}_{i_d,\ell} \neq 0$. 

To bound the probability that one of the $r$ columns is not recovered, note that each column of $\mA^{(k)}$ has at most $zn_k$ zeros, and each $(i_4,\ldots,i_d) \in Z_3$ is chosen uniformly at random from $[n_4] \times \cdots \times [n_d]$. Thus, for each $\ell \in [r]$ and each $(i_4,\ldots,i_d) \in Z_3$ \[\P\left\{\mA^{(4)}_{i_4,\ell} \cdots \mA^{(d)}_{i_d,\ell} \neq 0\right\} = \P\left\{\mA^{(4)}_{i_4,\ell} \neq 0 \wedge \cdots \wedge \mA^{(d)}_{i_d,\ell} \neq 0\right\} \ge \dfrac{(n_4-zn_4) \cdots (n_d-zn_d)}{n_4 \cdots n_d} = (1-z)^{d-3}.\] Then, since the elements of $Z_3$ are drawn without replacement, we have that for each $\ell \in [r]$, the probability that the $\ell$-th column isn't recovered is at most \[\P\left\{\mA^{(4)}_{i_4,\ell} \cdots \mA^{(d)}_{i_d,\ell} = 0 \ \forall (i_4,\ldots,i_d) \in Z_3\right\} = \prod_{(i_4,\ldots,i_d) \in Z_3}\P\left\{\mA^{(4)}_{i_4,\ell} \cdots \mA^{(d)}_{i_d,\ell} = 0\right\} \le (1-(1-z)^{d-3})^m.\] With a simple union bound over $\ell \in [r]$, we have that all columns of $\mA^{(1)}$ and $\mA^{(2)}$ are recovered and paired correctly with probability at least $1-r(1-(1-z)^{d-3})^m$. For the rest of the proof, we will assume that all $r$ columns of $\mA^{(1)}$ and $\mA^{(2)}$ were recovered and paired correctly. 


\subsubsection{Using slice-by-slice censored least squares to learn $\mA^{(3)},\ldots,\mA^{(d)}$}
\label{sec:nonadaptive_proof_CLS}
For each mode $k = 3,\ldots,d$, we use Algorithm~\ref{alg:CLS} on $m$ mode-$(1,2,k)$ subtensors of size $n_1 \times n_2 \times n_k$. Consider the subtensor $\calT_{:,:,i_3,\ldots,i_{k-1},:,i_{k+1},\ldots,i_d}$ for some $(i_3,\ldots,i_{k-1},i_{k+1},\ldots,i_d) \in Z_k$. Let \[\mC^{(k)} = \begin{bmatrix}(\mA^{(3)}_{i_3,1} \cdots \mA^{(k-1)}_{i_{k-1},1}\mA^{(k+1)}_{i_{k+1},1}\cdots \mA^{(d)}_{i_d,1})\mA^{(k)}_{:,1} & \cdots & (\mA^{(3)}_{i_3,r} \cdots \mA^{(k-1)}_{i_{k-1},r}\mA^{(k+1)}_{i_{k+1},r}\cdots \mA^{(d)}_{i_d,r})\mA^{(k)}_{:,r}\end{bmatrix},\] i.e., $\mA^{(k)}$, but with each column rescaled. Then, we can write 

\begin{align*}
\calT_{:,:,i_3,\ldots,i_{k-1},:,i_{k+1},\ldots,i_d} &= \sum_{\ell = 1}^{r}\mA^{(1)}_{:,\ell} \circ \mA^{(2)}_{:,\ell} \circ \mA^{(k)}_{:,\ell} \cdot (\mA^{(3)}_{i_3,\ell} \cdots \mA^{(k-1)}_{i_{k-1},\ell}\mA^{(k+1)}_{i_{k+1},\ell}\cdots \mA^{(d)}_{i_d,\ell}) 
\\
&= \sum_{\ell = 1}^{r}\mA^{(1)}_{:,\ell} \circ \mA^{(2)}_{:,\ell} \circ \mC^{(k)}_{:,\ell}
\\
&= \left[\left[\mA^{(1)},\mA^{(2)},\mC^{(k)}\right]\right]
\end{align*} 

Because we sampled each mode-$(1,2)$ slice of $\calT_{:,:,i_3,\ldots,i_{k-1},:,i_{k+1},\ldots,i_d}$ in locations $L \subset [n_1] \times [n_2]$ corresponding to $r$ linearly independent rows of $\mA^{(1)} \odot \mA^{(2)}$, the systems of equations that Algorithm~\ref{alg:CLS} uses to solve for each row of $\mA^{(k)}$ is full-rank, and thus, in line 15, Algorithm~\ref{alg:CLS} returns $\mC^{(k)}$, i.e. $\mA^{(k)}$, but with the $\ell$-th column rescaled. In line 17, we store all the columns which were rescaled by a non-zero factor, i.e. the ones for which $\mA^{(3)}_{i_3,\ell}\cdots\mA^{(k-1)}_{i_{k-1},\ell}\mA^{(k+1)}_{i_{k+1},\ell}\cdots\mA^{(d)}_{i_d,\ell} \neq 0$. In line 20, we update the set of columns that have not been recovered. Hence, after the end of each iteration of the for loop in line 9, we have recovered all columns (up to a non-zero scale factor) for which $\mA^{(3)}_{i_3,\ell}\cdots\mA^{(k-1)}_{i_{k-1},\ell}\mA^{(k+1)}_{i_{k+1},\ell}\cdots\mA^{(d)}_{i_d,\ell} \neq 0$ for at least one tuple $(i_3,\ldots,i_{k-1},i_{k+1},\ldots,i_d) \in Z_k$. With similar logic as in the previous subsection, after running Algorithm~\ref{alg:CLS} on the $|Z_k| = m$ mode-$(1,2,k)$ subtensors, the probability that all of the columns of $\mA^{(k)}$ are recovered up to a scale factor is at least $1-r(1-(1-z)^{d-3})^m$. 

For the last mode $(k = d)$, the same thing happens, except in line-19, we scale each column of the output of Algorithm~\ref{alg:CLS} by the appropriate scale factor based on the factor matrices $\mA^{(3)},\ldots,\mA^{(d-1)}$ that have already been recovered.

By using a simple union bound over the $d-3$ modes $k = 3,\ldots,d$, we have that all columns of the factor matrices $\mA^{(3)},\ldots,\mA^{(d)}$ are recovered with probability at least $1-(d-3)r(1-(1-z)^{d-3})^m$. Furthermore, the total number of samples taken from each mode-$(1,2,k)$ subtensor $\calT_{:,:,i_3,\ldots,i_{k-1},:,i_{k+1},\ldots,i_d}$ is $|L \times [n_k]| = |L| \cdot |[n_k]| = rn_k$, and for each mode $k = 3,\ldots,d$, we sample $|Z_k| = m$ such subtensors. So, the total number of samples used in this step is $\sum_{k = 3}^{d}mrn_k$.

\subsubsection{Putting it all together}

The total probability of our algorithm failing is the sum of the probabilities of (i) failing to complete one of the $sm$ mode-$(1,2)$ slices, (ii) failing to learn a column of $\mA^{(1)}$ and $\mA^{(2)}$, and (iii) failing to learn a column of one of $\mA^{(3)},\ldots,\mA^{(d)}$, which is \[sm\delta + r(1-(1-z)^{d-3})^m + (d-3)r(1-(1-z)^{d-3})^m = sm\delta + (d-2)r(1-(1-z)^{d-3})^m.\] 

Furthermore, the total number of samples required is the sum of the number of samples used in the slice completion and censored least squares steps, which is \[O(sm\mu_0 n_2 r \log(r/\delta))+smn_1r + mr\sum_{k = 3}^{d}n_k.\]

\subsection{Runtime of Algorithm~\ref{alg:adaptive_sampling}}
\label{sec:AdaptiveRuntime}
In this section, we outline the computational complexity of the steps of Algorithm~\ref{alg:adaptive_sampling}. 

{\bf Sample and complete $sm$ mode-$(1,2)$ slices}: In the noise-free setting, the matrix completion algorithm in \cite{BalcanZhang2016} works by iterating through the matrix $\mM \in \R^{n_1 \times n_2}$ one column at a time and attempting to learn a basis $\mB$ for the $r$-dimensional columnspace. Initially, the basis is empty. For each column $\mM_{:,i}$, the algorithm samples a random subset $\Lambda \subset [n_1]$ with $|\Lambda| = w$ entries (where $w = \min\{O(\mu_0r\log(r/\delta)),n_1\}$) and checks if $\mM_{\Lambda,i} \in \text{col}(\mB_{\Lambda},:)$, i.e., $\mM_{\Lambda,i} = \mB_{\Lambda,:}\mB_{\Lambda,:}^{\dagger}\mM_{\Lambda,i}$. If yes, the algorithm assumes that $\mM_{:,i} \in \text{col}(\mB)$ and completes the column using the formula $\mM_{:,i} = \mB\mB_{\Lambda}^{\dagger}\mM_{\Lambda,i}$. Otherwise, the algorithm samples the rest of the column $\mM_{:,i}$ and concatenates it to $\mB$, and then picks a new random subset $\Lambda \subset [n_1]$. The cost dominating steps of this algorithm are (i) updating $\mB$, $\Lambda$, $\mB_{\Lambda,:}^{\dagger}$ a total $r$ times which costs $O(r^2w)$ operations per update and (ii) computing $\mB_{\Lambda,:}\mB_{\Lambda,:}^{\dagger}\mM_{\Lambda,i}$ and/or $\mB\mB_{\Lambda,:}^{\dagger}\mM_{\Lambda,i}$ for every column which costs $O(rw)+O(rn_1) = O(rn_1)$ operations per column. Thus, the total cost of completing one slice is $n_2 \cdot O(rn_1) + r \cdot O(r^2w) = O(rn_1n_2 + r^3w)$ operations.

Hence, the cost of using the matrix completion algorithm in \cite{BalcanZhang2016} on $sm$ slices is $O(smrn_1n_2 + smr^3w) = O(smrn_1n_2 + smr^3\min\{\mu_0r\log(r/\delta),n_1\})$ operations.

{\bf Use Jennrich's Algorithm on $m$ mode-$(1,2,3)$ subtensors of size $n_1 \times n_2 \times s$}: Performing Jennrich's algorithm on a single tensor $\calX \in \R^{n_1 \times n_2 \times s}$ with CP-rank $r \le \min\{n_1,n_2\}$ requires: (i) $O(sn_1n_2)$ operations to compute the random linear combinations of the slices $\mX_{\vu} = \sum_{i_3 = 1}^{s}\vu_{i_3}\calX_{:,:,i_3}$ and $\mX_{\vv} = \sum_{i_3 = 1}^{s}\vv_{i_3}\calX_{:,:,i_3}$, (ii) $O(rn_1n_2)$ operations to compute the economical SVDs of the rank-$r$ matrices $\mX_{\vu}$ and $\mX_{\vv}$, (iii) $O(r(n_1+n_2)^2)$ operations to use these SVDs to compute $\mX_{\vu}\mX_{\vv}^{\dagger}$ and $\mX_{\vv}\mX_{\vu}^{\dagger}$, (iv) $O(r(n_1^2+n_2^2))$ operations to compute the $r$ eigenvalues and eigenvectors of $\mX_{\vu}\mX_{\vv}^{\dagger}$ and $\mX_{\vv}\mX_{\vu}^{\dagger}$. Finally, pairing the $r$ eigenvalues of $\mX_{\vu}\mX_{\vv}^{\dagger}$ and $\mX_{\vv}\mX_{\vu}^{\dagger}$ takes $O(r^2)$ operations. Hence, the total cost of performing Jennrich's algorithm on a single $n_1 \times n_2 \times s$ tensor is $O(r(n_1+n_2)^2+sn_1n_2)$. Therefore, the total cost of performing Jennrich's algorithm on $m$ mode-$(1,2,3)$ subtensors of size $n_1 \times n_2 \times s$ is $O(mr(n_1+n_2)^2+smn_1n_2)$.

{\bf Performing QR decomposition on $(\mA^{(1)} \odot \mA^{(2)})^T$}: Computing the Khatri-Rao product $\mA^{(1)} \odot \mA^{(2)}$ takes $O(rn_1n_2)$ operations and performing a QR decomposition to identify a subset $L' \subset [n_1n_2]$ of $r$ linearly independent rows corresponding to a subset $L \subset [n_1] \times [n_2]$ of indices $(i_1,i_2)$ takes $O(r^2n_1n_2)$ operations. So the total cost of this step is $O(rn_1n_2) + O(r^2n_1n_2) = O(r^2n_1n_2)$ operations.

{\bf Using Slice-by-Slice Censored Least Squares on $m$ mode-$(1,2,k)$ subtensors}: Inverting the $r \times r$ submatrix $(\mA^{(1)} \odot \mA^{(2)})_{L',:}$ once and storing it takes $O(r^3)$ operations. For each slice of a mode-$(1,2,k)$ subtensor, we have $r$ samples whose $(i_1,i_2)$ coordinates are from the same subset $L$. Hence, the $r \times r$ system of equations for every slice of all subtensors is $(\mA^{(1)} \odot \mA^{(2)})_{L',:}$. Since this inverse was already stored, solving this system for one slice of a mode-$(1,2,k)$ subtensor takes $O(r^2)$ operations. So the total number of operations for performing slice-by-slice censored least squares on one mode-$(1,2,k)$ subtensor is $n_k \cdot O(r^2) = O(r^2n_k)$. Therefore, $O(mr^2n_k)$ operations are needed to perform slice-by-slice censored least squares on $m$ mode-$(1,2,k)$ subtensors. Doing this for modes $k = 3,\ldots,d$ thus takes $O(mr^2\sum_{k = 3}^{d}n_k)$ operations. Thus, the total cost of inverting $(\mA^{(1)} \odot \mA^{(2)})_{L',:}$ and performing slice-by-slice censored least squares on $m$ mode-$(1,2,k)$ tensors for $k = 3,\ldots,d$ is $O(r^3 + mr^2\sum_{k = 3}^{d}n_k)$

Hence, the total number of operations required by Algorithm~\ref{alg:adaptive_sampling} (after discarding lower order terms) is \[O\left(r^2n_1n_2 + mr(n_1+n_2)^2+smrn_1n_2+mr^2\sum_{k = 3}^{d}n_k + smr^3\min\{\mu_0r\log(r/\delta),n_1\}\right).\] In the specific case where $s = O(1)$, $m = O(1)$, and $n_1 = n_2 = n_3 = \cdots = n_d = n$, the number of operations simplifies to $O(n^2r^2 + (d-2)nr^2)$. 

\pagebreak 

\subsection{Tensor Sandwich: Nonadaptive Independent Sampling}
\label{sec:GeneralAlgorithmNonadaptive}
\begin{algorithm}[ht]
\begin{algorithmic}[1]
\caption{Nonadaptive Tensor Deli for Order-$d$ tensors}
\label{alg:non_adaptive_dust_sampling}
    \State Pick any subset $S \subset [n_3]$ with $|S| = s$ indices
    \For{k = 3,\ldots,d}
        \State Generate a subset $Z_k \subset [n_3] \times \cdots \times [n_{k-1}] \times [n_{k+1}] \times \cdots \times [n_d]$ with $|Z_k| = m$ elements chosen uniformly at random without replacement.
    \EndFor
    \State Generate a random subset of sample locations $\Omega_M \subset [n_1] \times [n_2] \times S \times Z_3$ by independently including each entry with probability $c_0\dfrac{\mu_0 r\log^2(n_1+n_2)}{\min\{n_1,n_2\}}$.
    \For{$k = 3,\ldots,d$}
        \State Generate a random subset of sample locations $\Omega_k \subset [n_1] \times [n_2] \times [n_k] \times Z_k$ (slight misuse of notation) by independently including each entry with probability $c_3\dfrac{\mu_0^2 r^2\log n_k}{n_1n_2}$.
    \EndFor
    \For{$(i_1,\ldots,i_d) \in \Omega := \Omega_M \cup \bigcup_{k = 3}^{d}\Omega_k$}
        \State Sample $\calT_{i_1,\ldots,i_d}$
    \EndFor
    \For{$(i_4,\ldots,i_d) \in Z_3$}
        \For{$i_3 \in S$}
            \State Use nuclear norm minimization to complete $\calT_{:,:,i_3,i_4,\ldots,i_d}$ using only the sample locations $(i_1,\ldots,i_d) \in \Omega_M$.
        \EndFor
        \State Use Jennrich's algorithm on each completed subtensor $\calT_{:, :, S, i_4, \ldots, i_d}$ to recover the factor pairs $\mA^{(1)}_{:,\ell} \circ \mA^{(2)}_{:,\ell}$ for $\ell$ such that $\prod_{\substack{k = 3\\k \neq t}}^{d}\mA^{(k)}_{i_k,\ell} \neq 0$.
    \EndFor
    \For{$k = 3,\ldots,d$}
        \State $R = [r]$
        \For{$(i_3,\ldots,i_{k-1},i_{k+1},\ldots,i_d) \in Z_k$}
            \State $\mB = \text{CensoredLeastSquares}(\calT_{:,:,i_3,\ldots,i_{k-1},:,i_{k+1},\ldots,i_d},\Omega_k,\mA^{(1)},\mA^{(2)})$
            \If{$k \neq d$}
                \State $\mA^{(k)}_{:,R} = \mB_{:,R}$
            \Else
                \State $\mA^{(k)}_{:,R} = \left(\prod_{t = 3}^{d-1}\mA^{(t)}_{i_t,\ell}\right)^{-1}\mB_{:,R}$
            \EndIf
            \State $R = \{\ell : \mA^{(k)}_{:,\ell} = \vct{0}\}$
        \EndFor
    \EndFor
    \State \Return $\mA^{(1)},\ldots,\mA^{(d)}$
\end{algorithmic}
\end{algorithm}

\subsection{Proof of Theorem~\ref{thm:non_adaptive_dust_sampling}}

\subsubsection{Completing $sm$ slices}
In lines 9-11 above, we use nuclear norm minimization to complete $|S| \cdot |Z_3| = sm$ slices of the tensor which each have dimensions $n_1 \times n_2$. Since each slice $\calT_{:,:,i_3,\ldots,i_d}$ has $\text{col-span}(\calT_{:,:,i_3,\ldots,i_d}) \subseteq \text{col-span}(\mA^{(1)})$ and $\text{row-span}(\calT_{:,:,i_3,\ldots,i_d}) \subseteq \text{col-span}(\mA^{(2)})$, and the factor matrices satisfy $\mu(\mA^{(1)}) \le \mu_0$ and $\mu(\mA^{(2)}) \le \mu_0$, each slice satisfies the assumptions of \cite{chen2015incoherence}. Hence, for each slice $\calT_{:,:,i_3,\ldots,i_d}, i_3 \in S, (i_4,\ldots,i_d) \in Z_3$, with probability $1-c_1(n_1+n_2)^{-c_2}$, nuclear norm minimization completes the slice. By taking a simple union bound over all $sm$ slices, we have that with probability at least $1-c_1sm(n_1+n_2)^{-c_2}$, all $sm$ slices are completed. For the rest of the proof, we will assume that these $sm$ slices were successfully completed.

\subsubsection{Using Jennrich's algorithm on $m$ subtensors to learn $\mA^{(1)}$ and $\mA^{(2)}$}

In lines 9 and 12, we use Jennrich's algorithm on $|Z_3| = m$ mode-$(1,2,3)$ subtensors of size $n_1 \times n_2 \times s$. In the exact same way as in Section~\ref{sec:nonadaptive_proof_Jennrich}, all columns of $\mA^{(1)}$ and $\mA^{(2)}$ are recovered and paired correctly with probability at least $1-r(1-(1-z)^{d-3})^m$. For the rest of the proof, we will assume that all $r$ columns of $\mA^{(1)}$ and $\mA^{(2)}$ were recovered and paired correctly.

\subsubsection{Using slice-by-slice censored least squares to learn $\mA^{(3)},\ldots,\mA^{(d)}$}

Learning the factor matrices $\mA^{(3)},\ldots,\mA^{(d)}$ works the same way as in Section~\ref{sec:nonadaptive_proof_CLS}, except instead of observing $r$ carefully chosen samples in each mode-$(1,2)$ slice of a mode-$(1,2,k)$ subtensor, we observe random samples, which introduces an additional point of failiure for our algorithm. 

Again, for each mode $k = 3,\ldots,d$, we use Algorithm~\ref{alg:CLS} on $m$ mode-$(1,2,k)$ subtensors of size $n_1 \times n_2 \times n_k$. As before, we can express the subtensor $\calT_{:,:,i_3,\ldots,i_{k-1},:,i_{k+1},\ldots,i_d}$ for some $(i_3,\ldots,i_{k-1},i_{k+1},\ldots,i_d) \in Z_k$ as \[\calT_{:,:,i_3,\ldots,i_{k-1},:,i_{k+1},\ldots,i_d} = \left[\left[\mA^{(1)},\mA^{(2)},\mC^{(k)}\right]\right]\] where \[\mC^{(k)} = \begin{bmatrix}(\mA^{(3)}_{i_3,1} \cdots \mA^{(k-1)}_{i_{k-1},1}\mA^{(k+1)}_{i_{k+1},1}\cdots \mA^{(d)}_{i_d,1})\mA^{(k)}_{:,1} & \cdots & (\mA^{(3)}_{i_3,r} \cdots \mA^{(k-1)}_{i_{k-1},r}\mA^{(k+1)}_{i_{k+1},r}\cdots \mA^{(d)}_{i_d,r})\mA^{(k)}_{:,r}\end{bmatrix},\] i.e., $\mA^{(k)}$, but with each column rescaled.

In the $i_k$-th mode-$(1,2)$ slice of this subtensor, we have sampled each of the $n_1n_2$ entries with probability $c_3\tfrac{\mu_0^2r^2\log n_k}{n_1n_2}$. So to learn the $i_k$-th row of $\mA^{(k)}$, Algorithm~\ref{alg:CLS} solves a system of linear equations where the corresponding matrix is formed by keeping (or deleting) each row of $\mA^{(1)} \odot \mA^{(2)}$ with probability $c_3\tfrac{\mu_0^2r^2\log n_k}{n_1n_2}$. By Lemma~\ref{lem:coherence_KR} (proved in Appendix~\ref{sec:Lemmas}), the matrix $\mA^{(1)} \odot \mA^{(2)} \in \R^{n_1n_2 \times r}$ has coherence bounded by $\mu(\mA^{(1)} \odot \mA^{(2)}) \le \mu(\mA^{(1)})\mu(\mA^{(2)})r \le \mu_0^2r$. Then, by using Lemma~\ref{lem:random_subset_full_rank} (also proved in Appendix~\ref{sec:Lemmas}), the probability that this system has rank-$r$ is at least \[1-r\exp\left(-\dfrac{n_1n_2 \cdot c_3\tfrac{\mu_0^2r^2\log n_k}{n_1n_2}}{r\mu(\mA^{(1)} \odot \mA^{(2)})}\right) \ge 1-r\exp\left(-\dfrac{n_1n_2 \cdot c_3\tfrac{\mu_0^2r^2\log n_k}{n_1n_2}}{\mu_0^2r^2}\right) = 1-rn_k^{-c_3}.\] With a simple union bound, the probability that all $n_k$ systems of equations (one for each slice of $\calT_{:,:,i_3,\ldots,i_{k-1},:,i_{k+1},\ldots,i_d}$) have rank-$r$ is at least $1-rn_k^{-c_3+1}$. Then, the probability that all $n_k$ systems of equations for all $m$ subtensors $\calT_{:,:,i_3,\ldots,i_{k-1},:,i_{k+1},\ldots,i_d}$ for $(i_3,\ldots,i_{k-1},i_{k+1},\ldots,i_d) \in Z_k$ have rank-$r$ is at least $1-mrn_k^{-c_3+1}$. 

If for all $m$ mode-$(1,2,k)$ subtensors $\calT_{:,:,i_3,\ldots,i_{k-1},:,i_{k+1},\ldots,i_d}$, all $n_k$ systems of equations have rank-$r$, then all $m$ calls of Algorithm~\ref{alg:CLS} return a version of $\mA^{(k)}$ with each column rescaled. Conditioned on this event, the probability that all $r$ of the columns of $\mA^{(k)}$ are recovered is at least $1-r(1-(1-z)^{d-3})^m$. 

Therefore, the probability of $\mA^{(k)}$ being recovered (up to scaling each column by a different constant factor) is at least $1-r(1-(1-z)^{d-3})^m-mrn_k^{-c_3+1}$. With a union bound over all modes $k = 3,\ldots,d$, the probability that $\mA^{(3)},\ldots,\mA^{(d)}$ are all recovered is at least $1-(d-3)r(1-(1-z)^{(d-3)})^m-mr\sum_{k = 3}^{d}n_k^{-c_3+1}$.

\subsubsection{Putting it all together}
The total probability of our algorithm failing is the sum of the probabilities of (i) failing to complete one of the $sm$ mode-$(1,2)$ slices, (ii) failing to learn a column of $\mA^{(1)}$ and $\mA^{(2)}$, and (iii) failing to learn a column of one of $\mA^{(3)},\ldots,\mA^{(d)}$, which is \begin{align*}&sm\delta + r(1-(1-z)^{d-3})^m + (d-3)r(1-(1-z)^{d-3})^m + mr\sum_{k = 3}^{d}n_k^{-c_3+1} \\ =& sm\delta + (d-2)r(1-(1-z)^{d-3})^m + mr\sum_{k = 3}^{d}n_k^{-c_3+1}.\end{align*} 

The total number of samples is bounded by \[|\Omega| = \left|\Omega_M \cup \bigcup_{k = 3}^{d}\Omega_k\right| \le |\Omega_M| + \sum_{k = 3}^{d}|\Omega_k| =: N.\]

Since each of the sets $\Omega_M$ and $\Omega_k$ for $k = 3,\ldots,d$ is formed by sampling a region of the tensor independently with some constant probability, the cardinalities of $\Omega_M$ and $\Omega_k$ for $k = 3,\ldots,d$ follow a binomial distribution. Specifically,

\[|\Omega_M| \sim \text{Bin}\left(smn_1n_2,c_0\dfrac{\mu_0r\log^2(n_1+n_2)}{\min\{n_1,n_2\}}\right) \quad \text{and} \quad |\Omega_k| \sim \text{Bin}\left(mn_1n_2n_k,c_3\dfrac{\mu_0^2r^2\log n_k}{n_1n_2}\right).\]

Hence, the expected number of samples is bounded by \begin{align*}
\E|\Omega| &\le \E N
\\
&= \E|\Omega_M|+\sum_{k = 3}^{d}\E|\Omega_k| 
\\
&= smn_1n_2 \cdot c_0\dfrac{\mu_0 r\log^2(n_1+n_2)}{\min\{n_1,n_2\}} + \sum_{k = 3}^{d}mn_1n_2n_k \cdot c_3\dfrac{\mu_0^2r^2\log n_k}{n_1n_2}
\\
&= c_0sm\mu_0r\max\{n_1,n_2\}\log^2(n_1+n_2) + c_3m\mu_0^2r^2\sum_{k = 3}^{d}n_k \log n_k
\end{align*}

Furthermore, since $N = |\Omega_M|+\sum_{k = 3}^{d}|\Omega_k|$ is the sum of several independent indicator random variables (one for each possible sample location in $\Omega_M$ and $\Omega_k$ for $k = 3, \ldots, d$), it satisfies the standard Chernoff bound that \[\P\left\{N \ge (1+\rho)\E N\right\} \le \left(\dfrac{e^{-\rho}}{(1+\rho)^{1+\rho}}\right)^{\E N} \quad \text{for} \quad \rho > 0.\] By taking $\rho = e-1$, we have that the number of samples used satisfies \[|\Omega| \le N \le e\E N = c_0 e sm\mu_0 r\max\{n_1,n_2\}\log^2(n_1+n_2) + c_3 e m\mu_0^2r^2\sum_{k = 3}^{d}n_k \log n_k\] with probability at least $1-e^{-\E N}$.

\subsection{Runtime of Algorithm~\ref{alg:non_adaptive_dust_sampling}}
\label{sec:NonadaptiveRuntime}
In this section, we outline the computational complexity of the steps of Algorithm~\ref{alg:non_adaptive_dust_sampling}. 
\\
{\bf Sample tensor}: This takes $O(|\Omega|) = O(sm\mu_0r\max\{n_1,n_2\}\log^2(n_1+n_2)+m\mu_0^2r^2\sum_{k=3}^{d}n_k\log n_k)$ operations.
\\
{\bf Complete $sm$ mode-$(1,2)$ slices via nuclear norm minimization}: The nuclear norm minimization problem $\min_{\mX}\|\mX\|_{*}$ such that $\mX_{i,j} = \mM_{i,j}$ for $(i,j) \in \Omega$ can be solve in polynomial time using a semidefinite program \cite{goldfarb2009solving}.
\\
{\bf Use Jennrich's Algorithm on $m$ mode-$(1,2,3)$ subtensors of size $n_1 \times n_2 \times s$}: In exactly the same way as for our adaptive sampling algorithm, the total cost of performing Jennrich's algorithm on $m$ mode-$(1,2,3)$ subtensors of size $n_1 \times n_2 \times s$ is $O(mr(n_1+n_2)^2+smn_1n_2)$.
\\
{\bf Using Slice-by-Slice Censored Least Squares on $m$ mode-$(1,2,k)$ subtensors}: Fix one mode-$(1,2,k)$ subtensor. Let $\omega_{i_k}$ be the number of samples observed in the $i_k$-th slice of this mode-$(1,2,k)$ subtensor. Computing the $\omega_{i_k}$ rows of $\mA^{(1)}\odot \mA^{(2)}$ takes $O(r\omega_{i_k})$ operations and solving the resulting $\omega_{i_k} \times r$ system of equations takes $O(r^2\omega_{i_k})$ operations. So the total number of operations for performing slice-by-slice censored least squares on one mode-$(1,2,k)$ subtensor is $O(r^2\sum_{i_k = 1}^{n_k}\omega_{i_k}) = O(\mu_0^2r^4n_k\log n_k)$. Therefore, $O(m\mu_0^2r^4n_k\log n_k)$ operations are needed to perform slice-by-slice censored least squares on $m$ mode-$(1,2,k)$ subtensors. Doing this for modes $k = 3,\ldots,d$ thus takes $O(m\mu_0^2r^4\sum_{k = 3}^{d}n_k\log n_k)$ operations.
\\
Hence, Algorithm~\ref{alg:non_adaptive_dust_sampling} runs in polynomial time, and the total number of operations required is dominated by the cost of completing the $sm$ mode-$(1,2)$ densely sampled slices via nuclear norm minimization. 

\subsection{Proof of Theorem~\ref{thm:zerofree_adapt}}

Theorem~\ref{thm:zerofree_adapt} is a special case of Theorem~\ref{thm:adaptive_sampling}. Specifically, if for $k = 3,\ldots,d$, $\mA^{(k)}$ has no entries which are exactly $0$, then we can use Theorem~\ref{thm:adaptive_sampling} and Algorithm~\ref{alg:adaptive_sampling} with $z = 0$ and $m = 1$. Because there are no zeros in $\mA^{(3)},\ldots,\mA^{(d)}$, we can simply pick indices $i^*_3 \in [n_3], \ldots, i^*_d \in [n_d]$ and take each $Z_k$ to be the set with the single tuple $(i^*_3,\ldots,i^*_{k-1},i^*_{k+1},\ldots,i^*_d)$. With these choices, Theorem~\ref{thm:adaptive_sampling} and Algorithm~\ref{alg:adaptive_sampling} reduce to Theorem~\ref{thm:zerofree_adapt} and Algorithm~\ref{alg:zerofree_adapt}. 

\subsection{Proof of Theorem~\ref{thm:zerofree_nonadapt}}

Similarly, Theorem~\ref{thm:zerofree_nonadapt} is a special case of Theorem~\ref{thm:non_adaptive_dust_sampling}. Again, if for $k = 3,\ldots,d$, $\mA^{(k)}$ has no entries which are exactly $0$, then we can use Theorem~\ref{thm:non_adaptive_dust_sampling} and Algorithm~\ref{alg:non_adaptive_dust_sampling} with $z = 0$ and $m = 1$. We can again pick indices $i^*_3 \in [n_3], \ldots, i^*_d \in [n_d]$ and take each $Z_k$ to be the set with the single tuple $(i^*_3,\ldots,i^*_{k-1},i^*_{k+1},\ldots,i^*_d)$. With these choices, Theorem~\ref{thm:non_adaptive_dust_sampling} and Algorithm~\ref{alg:non_adaptive_dust_sampling} reduce to Theorem~\ref{thm:zerofree_nonadapt} and Algorithm~\ref{alg:zerofree_nonadapt}. 

\subsection{Proof of Theorem~\ref{thm:3mode_adapt}}
Theorem~\ref{thm:3mode_adapt} is also a special case of Theorem~\ref{thm:adaptive_sampling}. Specifically, if $\calT$ is an order $d = 3$ tensor, then we can use Theorem~\ref{thm:adaptive_sampling} and Algorithm~\ref{alg:adaptive_sampling} with $z = 1-\tfrac{1}{n_3}$ (because each column of $\mA^{(3)}$ needs to have at least one non-zero entry) and $m = 1$. With these choices, Theorem~\ref{thm:adaptive_sampling} and Algorithm~\ref{alg:adaptive_sampling} reduce to Theorem~\ref{thm:3mode_adapt} and Algorithm~\ref{alg:3mode_adapt}. 

\subsection{Proof of Theorem~\ref{thm:3mode_nonadapt}}
Theorem~\ref{thm:3mode_nonadapt} is also a special case of Theorem~\ref{thm:non_adaptive_dust_sampling}. Specifically, if $\calT$ is an order $d = 3$ tensor, then we can use Theorem~\ref{alg:non_adaptive_dust_sampling} and Algorithm~\ref{alg:non_adaptive_dust_sampling} with $z = 1-\tfrac{1}{n_3}$ (because each column of $\mA^{(3)}$ needs to have at least one non-zero entry) and $m = 1$. With these choices, Theorem~\ref{thm:non_adaptive_dust_sampling} and Algorithm~\ref{alg:non_adaptive_dust_sampling} reduce to Theorem~\ref{thm:3mode_nonadapt} and Algorithm~\ref{alg:3mode_nonadapt}.

\section{Misc Lemmas}
\label{sec:Lemmas}

\begin{lemma}
\label{lem:random_subset_full_rank}
Let $\mZ \in \R^{n \times r}$ with $\rank(\mZ) = r$. Suppose we generate a new matrix $\widetilde{\mZ}$ by keeping (or deleting) each row of $\mZ$ independently with probability $p$. Then, $\rank(\widetilde{\mZ}) = r$ with probability at least $1-r\exp\left(-\tfrac{np}{r\mu(\mZ)}\right)$. In particular, if $p \ge \dfrac{\mu(\mZ)r\log(r/\delta)}{n}$, then $\rank(\widetilde{\mZ}) = r$ with probability at least $1-\delta$. 
\end{lemma}

\begin{proof}
Let $\mZ = \mQ\mR$ be the QR-decomposition of $\mZ$, i.e. $\mQ \in \R^{n \times r}$ is orthonormal and $\mR \in \R^{r \times r}$. Let $\vz_i \in \R^r$ and $\vq_i \in \R^r$ denote the $i$-th rows of $\mZ$ and $\mQ$ respectively. Also, let $\mSigma = \mZ^T\mZ = \mR^T\mR \in \R^{r \times r}$. Note that since $\rank(\mZ) = r$, we also have $\rank(\mR) = r$, and thus, $\rank(\mSigma) = \rank(\mR^T\mR) = r$.

Then, we have \[\widetilde{\mZ}^T\widetilde{\mZ} = \sum_{i = 1}^{n}\xi_i\vz_i\vz_i^T\] where $\xi_i$ are i.i.d. $\text{Bernoulli}(p)$ random variables. Now consider the matrix \[\mSigma^{-1/2}\widetilde{\mZ}^T\widetilde{\mZ}\mSigma^{-1/2} = \sum_{i = 1}^{n}\xi_i\mSigma^{-1/2}\vz_i\vz_i^T\mSigma^{-1/2}.\] This is a sum of independent symmetric PSD matrices. The norm of each term is bounded by \begin{align*}&\|\xi_i\mSigma^{-1/2}\vz_i\vz_i^T\mSigma^{-1/2}\| \le \|\mSigma^{-1/2}\vz_i\vz_i^T\mSigma^{-1/2}\| = \vz_i^T\mSigma^{-1/2}\mSigma^{-1/2}\vz_i 
\\
= &\vz_i^T\mSigma^{-1}\vz_i = \vq_i^T\mR(\mR^T\mR)^{-1}\mR^T\vq_i = \|\vq_i\|_2^2 \le \dfrac{\mu(\mZ)r}{n},\end{align*} where the last inequality is due to the definition of coherence $\mu(\mZ) = \displaystyle\tfrac{n}{r}\max_{1 \le i \le n}\|\text{proj}_{\text{col}(\mZ)}\ve_i\|_2^2 = \tfrac{n}{r}\max_{1 \le i \le n}\|\vq_i\|_2^2$. Furthermore, the expectation of $\mSigma^{-1/2}\widetilde{\mZ}^T\widetilde{\mZ}\mSigma^{-1/2}$ is \begin{align*}
&\E\left[ \mSigma^{-1/2}\widetilde{\mZ}^T\widetilde{\mZ}\mSigma^{-1/2}\right] = \sum_{i = 1}^{n}\E\left[\xi_i\mSigma^{-1/2}\vz_i\vz_i^T\mSigma^{-1/2}\right] = \sum_{i = 1}^{n}\E[\xi_i]\mSigma^{-1/2}\vz_i\vz_i^T\mSigma^{-1/2} 
\\
= &\sum_{i = 1}^{n}p\mSigma^{-1/2}\vz_i\vz_i^T\mSigma^{-1/2} = p\mSigma^{-1/2}\left[\sum_{i = 1}^{n}\vz_i\vz_i^T\right]\mSigma^{-1/2} = p\mSigma^{-1/2}\mZ^T\mZ\mSigma^{-1/2} = p\mSigma^{-1/2}\mSigma\mSigma^{-1/2} = p\mId
\end{align*}
So by applying the matrix Chernoff inequality, we have that for any $\rho \in [0,1)$, \[\P\left\{\lambda_{\text{min}}(\mSigma^{-1/2}\widetilde{\mZ}^T\widetilde{\mZ}\mSigma^{-1/2}) \le 0\right\} \le \P\left\{\lambda_{\text{min}}(\mSigma^{-1/2}\widetilde{\mZ}^T\widetilde{\mZ}\mSigma^{-1/2}) \le (1-\rho)p\right\} \le r\left(\tfrac{e^{-\rho}}{(1-\rho)^{1-\rho}}\right)^{\tfrac{np}{r\mu(\mZ)}}.\]
Since this bound holds for all $\rho \in [0,1)$, we have that \[\P\left\{\lambda_{\text{min}}(\mSigma^{-1/2}\widetilde{\mZ}^T\widetilde{\mZ}\mSigma^{-1/2}) \le 0\right\} \le \lim_{\rho \to 1^{-}} r\left(\tfrac{e^{-\rho}}{(1-\rho)^{1-\rho}}\right)^{\tfrac{np}{r\mu(\mZ)}} = r\exp\left(-\tfrac{np}{r\mu(\mZ)}\right).\]
Finally, since $\mSigma$ is full-rank, we have that $\rank(\widetilde{\mZ}) = r$ is equivalent to $\lambda_{\text{min}}(\mSigma^{-1/2}\widetilde{\mZ}^T\widetilde{\mZ}\mSigma^{-1/2}) > 0$, i.e. the complement of the event above. Thus, the probability that $\rank(\widetilde{\mZ}) = r$ is at least $1-r\exp\left(-\tfrac{np}{r\mu(\mZ)}\right)$.
\end{proof}

\begin{lemma}
\label{lem:coherence_KR}
If $\mA \in \R^{n_1 \times r}$ and $\mB \in \R^{n_2 \times r}$, then \[\mu(\mA \odot \mB) \le \mu(\mA)\mu(\mB)r.\] 
\end{lemma}

\begin{proof}
Let $\mA = \mU\mX$, $\mB = \mV\mY$, and $\mA \odot \mB = \mW\mZ$ be the QR-decompositions of $\mA$, $\mB$, and $\mA \odot \mB$ respectively. Then, since $\mW\mZ = \mA \odot \mB = (\mU\mX) \odot (\mV\mY) = (\mU \otimes \mV)(\mX \odot \mY)$, we have that $\text{span}(\text{col}(\mW)) \subset \text{span}(\text{col}(\mU \otimes \mV))$. Hence, \[\|\text{proj}_{\mW}\ve_i\|_2^2 \le \|\text{proj}_{\mU \otimes \mV}\ve_i\|_2^2 = \|(\mU \otimes \mV)^T\ve_i\|_2^2 = \|\mU^T\ve_{i'}\|_2^2 \cdot \|\mV^T\ve_{i''}\|_2^2 \le \dfrac{\mu(\mA)r}{n_1} \cdot \dfrac{\mu(\mB)r}{n_2}\] for all $i$, and thus, \[\mu(\mA \odot \mB) = \dfrac{n_1n_2}{r}\max_{1 \le i \le n_1n_2}\|\text{proj}_{\mW}\ve_i\|_2^2 \le \dfrac{n_1n_2}{r} \cdot \dfrac{\mu(\mA)r}{n_1} \cdot \dfrac{\mu(\mB)r}{n_2} = \mu(\mA)\mu(\mB)r.\]
\end{proof}

\bibliographystyle{plain}
\bibliography{ref}
\end{document}